\documentclass[10pt]{amsart}

\usepackage{amsfonts, amsmath, amsthm, amssymb, epsfig, bm, graphics,
  psfrag, latexsym, chngcntr, color, yfonts, mathtools, subfig, longtable,enumitem}
\usepackage[color=blue!20!white,textsize=tiny]{todonotes}
\usepackage[all,cmtip]{xy}
\usepackage{tikz} \usetikzlibrary{matrix,arrows}
\definecolor{darkblue}{rgb}{0,0,0.4}
\usepackage[colorlinks=true, citecolor=darkblue, filecolor=darkblue, linkcolor=darkblue, urlcolor=darkblue]{hyperref}
\usepackage[all]{hypcap}
\usepackage{tikz,ifthen}
\usetikzlibrary{calc,positioning,intersections}

\newtheorem{thm}{Theorem}[section]         
               
\newtheorem{cor}[thm]{Corollary}               
\newtheorem{prop}[thm]{Proposition}  
\newtheorem{ques}[thm]{Question}     
\newtheorem{prob}[thm]{Problem}     
\newtheorem{defn}[thm]{Definition} 
\newtheorem{rem}[thm]{Remark}
 
\newtheorem{alg}[thm]{Algorithm} 
\counterwithin{figure}{section}
\counterwithin{equation}{section}

\newcommand{\R}{\mathbb{R}}
\newcommand{\Z}{\mathbb{Z}}
\newcommand{\F}{\mathbb{F}}

\newcommand{\N}{\mathbb{N}}

\newcommand{\mc}{\mathcal}

\newcommand{\ol}{\overline}

\newcommand{\del}{\partial}
\newcommand{\sbs}{\subset}

\newcommand{\sm}{\setminus}

\newcommand{\Si}{\Sigma}

\newcommand{\al}{\alpha}
\newcommand{\be}{\beta}
\newcommand{\ga}{\gamma}

\newcommand{\KKnots}{K(\mathrm{links},\ast)}
\newcommand{\mrKKnots}{K(\mathrm{links},\ast)}

\newcommand{\nbd}{\mathrm{nbd}}

\newcommand{\CFK}{\mathit{CFK}}

\newcommand{\hatHFK}{\widehat{\mathit{HFK}}}
\newcommand{\hatCFK}{\widehat{\mathit{CFK}}}
\newcommand{\tCF}{\widetilde{\mathit{CF}}}
\newcommand{\tSF}{\widetilde{\mathit{SF}}}
\newcommand{\tHF}{\widetilde{\mathit{HF}}}

\newcommand{\taubot}{\tau_{\mathrm{bot}}}
\newcommand{\tautop}{\tau_{\mathrm{top}}}

\begin{document}
\title{Murasugi sum and  extremal knot Floer homology} \author{Zhechi Cheng}
\address{School of Mathematics and Statistics, Wuhan University, Wuhan, Hubei, China}
\email{\href{mailto:chengzhechi@gmail.com}{chengzhechi@gmail.com}}
\author{Matthew Hedden} \address{Department of Mathematics, Michigan
  State University, East Lansing, MI}
\email{\href{mailto:matthew.hedden@gmail.com}{matthew.hedden@gmail.com}}
\author{Sucharit Sarkar} \address{Department of Mathematics, University of California, Los Angeles, CA}
\email{\href{mailto:sucharit.sarkar@gmail.com}{sucharit.sarkar@gmail.com}}

\subjclass[2020]{57K18, 57R58}
\keywords{}

\date{}

\begin{abstract} The aim of this paper is to study the behavior of
  knot Floer homology under Murasugi sum.  We establish a graded
  version of Ni's isomorphism between the extremal knot Floer
  homology of Murasugi sum of two links and the tensor product of the
  extremal knot Floer homology groups of the two summands. We
  further prove that $\tau=g$ for each summand if and only if $\tau=g$
  holds for the Murasugi sum (with $\tau$ and $g$ defined
  appropriately for multi-component links).  Some applications are
  presented.
\end{abstract}

\maketitle

\section{Introduction}
The {\em Murasugi sum} is an operation that one can perform on isotopy
classes of surfaces with non-empty boundary embedded in 3-manifolds.
Applying it to Seifert surfaces yields an operation on isotopy classes
of links.  As the name suggests, the operation was introduced by
Murasugi \cite{Murasugi63,MurasugiAlternating}, whose motivation was a
calculation of the genus of an alternating link by means of the degree
of its Alexander polynomial.  An important point to be made about the
Murasugi sum is that it {\em not} a well-defined binary operation on
the set of links.  Indeed, many essential choices are made in its
definition; not only the isotopy classes of the chosen Seifert
surfaces, but also the polygons embedded therein along which the sum
is performed.  Despite these choices, one can easily show that the
coefficient of the Alexander polynomial corresponding to the first
betti number of the surfaces in question is multiplicative under
Murasugi sum.  Gabai later showed that Seifert genus behaves
additively under Murasugi sum, extending the well-known special case
of the additivity of genus under connected sum.  Indeed, he showed that
the Murasugi sum of two surfaces is minimal genus if and only the two
summands are minimal genus \cite{Gabai,GabaiII}.
Note, however, that pathology arises if one considers non-minimal
genus surfaces; for instance, Thompson showed that one can sum two unknots along genus one
surfaces to get a trefoil or sum two figure eights to get the unknot
\cite{Thompson}, and Able and Hirasawa have recently shown that in fact {\em any} knot can be obtained as a Murasugi sum of {\em any other two} knots along (typically) non-minimal genus Seifert surfaces \cite{AbleHirasawa}.

In light of the connections between the knot Floer homology groups and
both the Alexander polynomial (through their Euler characteristic
\cite{POZSzknotinvariants,JR}) and the genus (through their breadth
\cite{POZSzgenusbounds}), one might wonder about the behavior
of the extremal knot Floer homology under Murasugi sum.  
Here, ``extremal''
refers to the knot Floer homology group in Alexander grading given by
negative the genus of the surfaces used. (Modulo a
well-understood grading shift, this is isomorphic to the knot Floer
homology group in Alexander grading given by the genus, a group often referred to as the ``top" group)
For
instance, the knot Floer homology groups of a connected sum are a
bigraded tensor product of those of the summands \cite[Theorem
7.1]{POZSzknotinvariants}.  Simple examples exploiting the
non-uniqueness of Murasugi sum show that this cannot hold for general
Murasugi sums and, in fact, there can be no closed formula for the
knot Floer homology of the Murasugi sum of links in terms of the knot
Floer homology of the summands.  Despite this, it would be reasonable
to conjecture that the extremal knot Floer homology of a
Murasugi sum is a tensor product of the extremal terms of its summands. 
Ni proved that this is indeed the case for {\em ungraded} knot Floer
homology groups with field coefficients~\cite[Theorems 1.1,4.5]{YN06}
(cf. ~\cite[Corollary 8.8]{AJ}).  Since an ungraded vector space
over a field is determined up to isomorphism by its dimension, this
result is equivalent to saying that the rank of the 
extremal knot
Floer homology is multiplicative under Murasugi sum. It is
natural to wonder if Ni's result extends in a (Maslov) graded fashion,
and our first result confirms that this is indeed the case.  To state
it, we define the index of a (possibly disconnected) surface $R$ to be the quantity
$\frak{i}(R):=\frac {|\partial R|-\chi(R)}{2}$
\begin{thm}\label{thm1}
  For $i\in\{1,2\}$, let $L_i$ be an $l_i$-component link and let
  $R_i$ be a  Seifert surface for $L_i$. Let $R_1*R_2$ be a
  Murasugi sum of $R_1$ and $R_2$, and let $\del(R_1*R_2)=L=L_1*L_2$
  be the corresponding $l$-component link.  Then with
  $\F_2$-coefficients, we have a graded isomorphism
  \[
  \hatHFK(L,-\frak{i}(R_1*R_2))[l-1]\cong\hatHFK(L_1,-\frak{i}(R_1))[l_1-1]\otimes\hatHFK(L_2,-\frak{i}(R_2))[l_2-1].
  \]
\end{thm}

We should note that it is not clear how to extend Ni's argument, nor
the argument using the decomposition theorem for sutured Floer
homology presented by Juhasz, to yield the graded statement given
above (despite some effort to do so).  On a superficial level, though,
our proof follows the same strategy as its antecedents; namely, we
find particular Heegaard diagrams adapted to Seifert surfaces and
their Murasugi sum, and then analyze the resulting Floer complexes in
detail.  The diagrams we end up using are more specialized, however,
and yield more control over the combinatorics and homotopy theoretic
aspects of the associated chain complexes.  This increase in control
further allows us to glean some information about the rest of the
knot Floer homology filtration, in the form of the following result
about the integer-valued concordance invariant
$\tau$~\cite{POZSz4ballgenus} (for the extension of $\tau$ to links,
see \cite{Cavallo,GridBook,HeddenRaoux}).

\begin{thm}\label{thm2}
  If the link $L$ is the Murasugi sum of links $L_1$ and $L_2$ along
  minimal index Seifert surfaces, then $\tau(L_i)=g(L_i)$ for all
  $i\in\{1,2\}$ if and only if $\tau(L)=g(L)$. (Here $\tau$ denotes
  $\tau_{\mathrm{top}}$ from \cite{HeddenRaoux} when discussing links
  with more than one component.)
\end{thm}

A large class of links for which $\tau(L)=g(L)$ is provided so-called
{\em strongly quasipositive} links.  These links possess a Seifert
surface which is properly isotopic into the 4-ball onto a piece of an
algebraic curve and which therefore minimizes the smooth 4-genus.
Rudolph gave a partial extension of Gabai's results to 4-genera, by
showing that the Murasugi sum of links along Seifert surfaces is
strongly quasipositive if and only if the two summands are.  Our
result strengthens the resulting implications for the 4-genus.

Our results lead to topological restrictions on which link types can
be expressed as Murasugi sums of others along minimal index Seifert
surfaces.  Some of the complexity of this problem, and the restrictions
offered by our theorems, can be algebraically distilled by defining a
Grothendieck group of links.  Recall that the Grothendieck
group $K(M)$, of a commutative monoid $M$ is the quotient of the free
abelian group on the set $M$ by the relations $[x+y]=[x]+[y]$, where
on the left ``$+$'' is taken with respect to the monoidal operation and
on the right within the free abelian group.  While the Murasugi sum $\ast$
is not a monoidal operation on links (relying as it does on the choice of Seifert surface and embedded $2n$-gon), we can nonetheless define a group \[\KKnots= \frac{\Z\langle \{\text{Isotopy classes of links}\}\rangle}{[L_1\ast L_2]=[L_1]+[L_2]}\] which we call the {\em
  Grothendieck group of links under Murasugi sum along minimal index
  surfaces}.  It is simply the quotient of the free abelian group on the set of isotopy classes of links by the relations $[L_1\ast L_2]=[L_1]+[L_2]$, where $\ast$ denotes any Murasugi sum along any $2n$-gon in any minimal index Seifert surface for the links in question.   Thus $\KKnots$ consists of equivalence classes of links, where two
links are equivalent if they become isotopic after iteratively Murasugi summing both of them with some collection $(R_1,\ldots, R_i)$ of minimal index Seifert surfaces
(along any $2n$-gons embedded therein, and in any order).  Fibered links, endowed with
their (unique) minimal index Seifert surface, form an important class of links which is closed under Murasugi sums by Gabai's work \cite{Gabai} (see also \cite{Stallings} for the closure under plumbing).  If one considers their
associated Grothendieck subgroup 
$K(\mathrm{fibered\ links},\ast)<\KKnots$, a deep theorem arising from
the Giroux correspondence asserts that
$K(\mathrm{fibered\ links},\ast)\cong \Z\oplus\Z$, generated by the
positive and negative Hopf links \cite{GirouxGoodman}.  One might hope
that all links could similarly be generated by a small family, given
the complexity allowed by choices of Seifert surfaces and embedded
$2n$-gons.

Multiplicativity of the rank of the extremal knot Floer homology
under Murasugi sum shows that $\KKnots$ is infinitely generated.
Indeed, if we consider the rank of the top group as map from the set of links to the natural numbers $\N^\times$, viewed as a multiplicative monoid,  then its multiplicativity under Murasugi sums implies that this map descend to a group homomorphism  $\KKnots \to
K(\N^\times)\cong \mathbb{Q}_{>0}^\times$.  Non-trivial twist knots
have top group of rank equal to the number of twists, showing that the
map to $\N^\times$ is surjective, hence the map to
$\mathbb{Q}_{>0}^\times$ is surjective as well. $\KKnots$ is therefore
infinitely generated as an abelian group.  One could still hope,
however, that some simple infinite family of knots such as twist knots
generates all knots under Murasugi sum and de-summing.  Our result
dashes this hope, and indicates that $\KKnots$ is quite complicated.

\begin{cor}\label{cor3} The Poincar{\'e} polynomial of the top group of knot Floer homology induces a homomorphism
\[ P:\mrKKnots\rightarrow \mathbb{Q}^\times_{>0}(t),\]
where the codomain is the multiplicative group of rational functions in $t$ with positive rational coefficients.
\end{cor}
\noindent It would be interesting to identify the image of $P$, a problem in the realm of geography questions for knot Floer homology.   In particular, we have the following natural question:
\begin{ques}\label{geography} Is every Laurent polynomial with $\mathbb{N}$ coefficients realized as the Poincar{\'e} polynomial of the top group of knot Floer homology  for some link in the 3-sphere?
\end{ques}
\noindent Obstructions for a bigraded collection of abelian groups to arise as knot Floer homology groups were obtained in \cite{HeddenWatson,BaldwinVelaVick}, but these place no restriction on the top group.

Despite a lack of understanding of the geography question for the top group of knot Floer homology, our results indicate that any collection of knots whose Poincar{\'e} polynomials are coprime are linearly independent in the Grothendieck group, even if their total rank is the same. In particular, the kernel of the homomorphism  $\mathbb{Q}^\times_{>0}(t)\rightarrow \mathbb{Q}^\times_{>0}$ induced by setting $t$ equal $1$ intersects the image of $P$ non-trivially.  Perhaps more concretely, we have 
\begin{cor}\label{cor5}  Suppose the Poincar{\'e} polynomial of the top group of knot Floer homology of a link $L\subset S^3$ is irreducible, viewed as a Laurent polynomial over $\Z$.  If $L$ is a Murasugi sum of links $L=L_1\ast L_2$, then one of $L_i$ is fibered. 
\end{cor}
As another quick corollary, we can show that alternating links or, more generally, links with thin Floer homology, are far from generating all links under Murasugi sum.
\begin{cor}\label{cor6}  Suppose the top group of the knot Floer homology of $L\subset S^3$ is non-trivial in more than one Maslov grading.  Then $L$ is not a Murasugi sum of alternating links nor is any link which contains $L$ as a Murasugi summand.
\end{cor}
\noindent For instance, the top group of the  Kenoshita-Terasaka knot and its mutant, the Conway knot, have Poincar{\'e} polynomials given by $1+t$, up to multiplication by $t^k$ (with $k=2$ for the KT knot and $k=3$ for the Conway knot) \cite[Theorems 1.1 and 1.2]{POZSzmutation}.  Therefore neither can be realized as a Murasugi sum of alternating links, nor is there any way to iteratively   Murasugi sum them with other links to eventually arrive at a Murasugi sum of alternating (or thin) links.    

As a final corollary, our results can be used in conjunction with the literature to calculate the top group of an arbitrary cable knot:

\begin{cor}  Let $K_{p,q}$ be the $(p,q)$ cable of a knot $K$ with Seifert genus $g$.  Then for any $p>0$, we have  
\begin{enumerate}
\item If $q>0$, then $\hatHFK_*(K_{p,q},pg+\frac{(p-1)(q-1)}{2})\cong \hatHFK_*(K,g)$
\item If $q<0$, then $\hatHFK_*(K_{p,q},pg+\frac{(p-1)(q-1)}{2})\cong \hatHFK_{*-(p-1)(2g-q-1)}(K,g)$
\end{enumerate}
\end{cor}
\noindent The key observation, due to Neumann and Rudolph, is that $K_{p,q}\cong K_{p,\mathrm{sign}(q)}\ast T_{p,q}$, where $\mathrm{sign}(q)$ is $\pm 1$ depending on whether $q$ is positive or negative \cite[Figure 4.2]{NeumannRudolph}.  Since $T_{p,q}$ is fibered, our main result indicates that the top group of a cable knot is isomorphic to that of $K_{p,\mathrm{sign}(q)}$ shifted by the grading of the top group of the corresponding torus knot.  As the latter is well known to be $0$ if $q>0$ and $\frac{(p-1)(-q-1)}{2}$ if $q<0$, the corollary can then be deduced if the top group is known for two particular examples of $K_{p,q}$; one with $q$ positive, and one with $q$ negative.  But the results of \cite{cabling,cablingII} (cf. \cite{mattthesis}) indicate that the top group of $K_{p,pn+1}$  is isomorphic to that of $K$ and the bottom group of $K_{p,-pn+1}$ is isomorphic to that of $K$, provided in both cases that $n\gg0$.  Together with the symmetry between the top and bottom groups of knot Floer homology, and the observations above, the corollary follows.  This is essentially the argument for the special case of fibered cable knots from \cite{cablingcontactcomplex}.

We conclude this introduction by highlighting a few problems and questions raised by our work.  Perhaps the most interesting and challenging is
\
\begin{prob}  Determine the isomorphism type of $\KKnots$.  
\end{prob}
\noindent Solving this would, ideally, yield an explicit presentation for $\KKnots$ by generators and relations. 
Note that Gabai's work implies that the link invariant $b^{min}_1$ obtained by minimizing the first Betti number over all Seifert surfaces for a given link, is additive under Murasugi sums. Hence, it descends to a homomorphism $B^{min}_1:\KKnots\rightarrow K(\mathbb{N}^+)\cong \mathbb{Z}$.  An affirmative answer to the following question would solve the problem:
\begin{ques} Is the homomorphism  $P\oplus B^{min}_1: \KKnots\rightarrow \mathbb{Q}^\times_{>0}(t)\oplus \mathbb{Z}$ an isomorphism?
\end{ques}
\noindent Note that an affirmative answer would require an affirmative answer to the geography problem raised by Question \ref{geography}.  Moreover, combined with any of the known algorithms to compute knot Floer homology (e.g. \cite{MOS,Beliakova,POZSz-algmatchHFK}), one would also arrive at a solution to the isomorphism problem in $\KKnots$ and, presumably,  a presentation. While we are inclined to believe the answer is no, the restriction of $P\oplus B^{min}_1$ to the subgroup generated by fibered links {\em is}  an isomorphism onto its image.  Indeed, the image of $P$ on the fibered subgroup is the multiplicative subgroup $\{t^n\}_{n\in \Z}$ and the power $n$ associated to a given fibered link is the Hopf invariant of the 2-plane field associated to its corresponding open book decomposition (up to normalization, the Hopf invariant is equal to Rudolph's {\em enhancement} of the Milnor number \cite{RudolphEnhancement}). 
To conclude with a more tractable question, we leave the reader with:

\begin{ques} Does the Poincar{\'e} polynomial homomorphism $P: \KKnots\rightarrow \mathbb{Q}^\times_{>0}(t)$ contain an infinite rank subgroup in the kernel of the rank homomorphism obtained by setting $t=1$ in the Poincar{\'e} polynomial?
\end{ques}

\noindent {\bf Outline:} The paper is organized as follow: In Section~\ref{sec:2}, we review
the knot Floer homology for links, recall the definition for Murasugi
sum, construct Heegaard diagrams associated to Seifert surfaces, and
describe the Murasugi sum operation in terms of Heegaard diagrams. In
Section~\ref{sec:3}, we study some local isotopies on Heegaard
diagrams which will largely reduce the number of generators; moreover,
we prove that there is a subcomplex that remains unchanged when
applying these isotopies if some technical conditions are
satisfied. In Section~\ref{sec:4}, we use the simplifications from
the previous section to prove the main results.

 \subsection*{Acknowledgement} This article started almost fifteen
 years ago, but was placed on hiatus multiple times due to a variety
 of reasons. We are grateful to Robert Lipshitz, Yi
 Ni, Peter Ozsv\'ath, and Zolt\'an Szab\'o for many helpful conversations about this
 paper at various points in the past decade.  During this period,
 Zhechi Cheng was supported from NSFC grant No.~12126101, NSF
 grants DMS-1609148, DMS-1564172 and Swedish Research Council under grant No.~2016-06596, Matthew Hedden was supported from
 NSF grants DMS-0706979, DMS-0906258, CAREER DMS-1150872, DMS-1709016, DMS-2104664
 and an Alfred P. Sloan Research Fellowship and Sucharit Sarkar was
 supported from Clay Research Fellowship and NSF grants CAREER
 DMS-1350037, CAREER DMS-1643401, and DMS-1905717.

\section{Heegaard diagrams adapted to Seifert surfaces}\label{sec:2}

\subsection{Heegaard diagrams}\label{subsec:intro}
We begin with a quick review of Heegaard diagrams. Most of what follows extends in a straightforward manner to arbitrary closed connected oriented three-manifolds, but since
we are primarily concerned with the operation of Murasugi sum in $S^3$
we will specialize our definitions and constructions to this situation.  We begin by recalling the definition of a Heegaard diagram: 

\begin{defn} A \emph{Heegaard diagram for $S^3$} is a $4$-tuple
\[\mc{H}=(\Sigma_{(g)},\alpha^{(g+k-1)},\allowbreak\beta^{(g+k-1)},\allowbreak
w^{(k)})\] where 
\begin{itemize}[leftmargin=*]
\item $\Sigma\sbs S^3$ is an oriented surface of genus $g$ whose complement has two components, the closures of which are  genus
$g$ handlebodies $U_{\al}$ and $U_{\be}$ with
$\Si=\del U_{\al}=-\del U_{\be}$;
\item
$\alpha^{(g+k-1)}=(\alpha_1,\ldots,\alpha_{g+k-1})$ (respectively,
$\beta^{(g+k-1)}=(\beta_1,\ldots,\beta_{g+k-1})$) is a collection of disjoint
simple closed curves on $\Sigma$, each bounding a disk in the
handlebody $U_{\al}$ (respectively, $U_{\be}$), such that
$\Sigma\setminus\alpha$ (respectively, $\Sigma\setminus\beta$) has
exactly $k$ components;
\item the $\alpha$ circles are transverse
to the $\beta$ circles;
\item $w=(w_1,\ldots,w_k)$ is a collections of
markings on $\Si$, such that each component of $\Sigma\setminus\alpha$
contains a $w$ marking, and each component of $\Sigma\setminus\beta$
contains a $w$ marking.
\end{itemize} 
\end{defn} 
\noindent Unless otherwise mentioned, we will assume our
Heegaard diagrams to satisfy a certain technical condition called (weak)
\emph{ admissibility} \cite[Definition 3.5]{POZSzlinkinvariants} (cf. \cite[Definition 4.10]{POZSz}). A \emph{generator} is a $(g+k-1)$-tuple
$x=(x_1,\ldots,x_{g+k-1})$ of points in $\Sigma$, called the
\emph{coordinates} of $x$, such that each $\al$ and $\be$ circle contain exactly one of
the coordinates;
we will denote the set of generator by $\mc{G}_{\mc{H}}$.\\

 Let $L\sbs S^3$ be an $l$-component link and $R$ a Seifert
surface for $L$, which we assume to be oriented but not necessarily connected.   We have the following notion of a diagram adapted to $R$ \cite{OSz-hf-applications, YN06, AJ, MHAJSS},
  \begin{defn}  A \emph{Heegaard diagram adapted to $R$} is a $6$-tuple
\[\mc{H}=(\Sigma_{(g)},\alpha^{(g+k-1)},\beta^{(g+k-1)},z^{(k)},w^{(k)},S)\]
satisfying
\begin{itemize}[leftmargin=*]
\item $(\Sigma,\al,\be,w)$ and $(\Si,\al,\be,z)$ are both Heegaard diagrams
for $S^3$; 
\item $S\sbs \Si$ is an oriented surface-with-boundary which is isotopic to $R$ in $S^3$;
\item  each generator has at most $(k-\chi(R))$ coordinates inside $R$; 
\item the $2k$ markings $z=(z_1,\ldots,z_k)$ and $w=(w_1,\ldots,w_k)$ all lie on
$\del S$; 
\item each component of $\del S$ contains at least one marking,
and on each component of $\del S$, the $z$ markings and the $w$
markings alternate; 
\item the oriented arcs in $\del S$ joining each $z$
marking to the next $w$ marking are disjoint from the $\alpha$
circles, and the arcs in $\del S$ joining each $w$ marking to the next
$z$ marking are disjoint from the $\beta$ circles.
\end{itemize}
\end{defn}

Given a Seifert surface $R$ for an $l$-component link $L\sbs S^3$, we
can employ the following slightly enhanced version of the
algorithm from \cite{MHAJSS}, or a further modification thereof, to
construct a Heegaard diagram adapted to $R$.
\vskip0.1in
\begin{alg}{\em Adapting a Heegaard diagram to a Seifert surface} $R\subset S^3$.

\begin{enumerate}[leftmargin=*,label=(H-\arabic*),ref=H-\arabic*]
\item Embed a graph $G$ with $n$ vertices and $(n-\chi(R))$ edges
  in the interior of the surface $R$, such that $R$ deform retracts to
  $G$. Therefore, $R$ is isotopic to $\ol{\nbd_R(G)}$, the closure of
  a regular neighborhood of $G$ in $R$. This is essentially a band
  presentation of $R$. 
\item Consider $\ol{\nbd_{S^3}(G)}$, the closure of a regular
  neighborhood of $G$ in $S^3$. Although $\ol{\nbd_{S^3}(G)}$ is a
  union of handlebodies, its complement in $S^3$ is usually not. 
  Rectify this by tunneling out some one-handles from the complement
  and adding them to $\ol{\nbd_{S^3}(G)}$, so as to get a Heegaard
  decomposition of $S^3$.
\item Let $U_{\al}$ be the handlebody obtained from
  $\ol{\nbd_{S^3}(G)}$ by adding these new handles, and let $U_{\be}$
  be complementary handlebody. Let $\Si$ be the dividing Heegaard surface,
  oriented as the boundary of $U_{\al}$. 
\item Push off $\ol{\nbd_R(G)}$
  towards $\Si$ to get a surface $S\sbs \Si$ in a way so as to ensure
  that the orientation on $S$ induced by $R$ agrees with the one
  induced by $\Si$. 
\item Place $2k$ distinct markings
  $z=(z_1,\ldots,z_k)$ and $w=(w_1,\ldots,w_k)$ on $\del S$ such that
  each component of $\del S$ contains at least one $z$ and $w$ marking, and on
  each component of $\del S$, the $z$ markings and the $w$ markings
  alternate. 
\item\label{item:construct-al-be-on-diagram} If the surface $\Si$ has genus $g$, then  draw
  $(g+k-1)$ $\al$ circles and $(g+k-1)$ $\be$ circles on $\Si\sm(z\cup
  w)$ such that the following holds: 
  \begin{enumerate}[leftmargin=*]
  \item The $\al$ circles are disjoint from one another.
  \item The $\be$ circles are disjoint from one another.
  \item The $\al$ circles intersect the $\be$ circles transversally.
  \item Each component of $\Si\sm\al$ contains one $z$ marking and one
    $w$ marking
  \item Each component of $\Si\sm\be$ contains one $z$ marking and one
    $w$ marking.
  \item\label{item:alpha-max-intersect} Exactly $(k-\chi(R))$ $\al$ circles intersect $S$.
  \end{enumerate}
\item\label{item:finger-moves-heegaard-diagram}  From each $w$ marking, as one travels along $-\del S$
  to the next $z$ marking,  isotope all the $\al$ circles that one
  encounters, by finger moves, across the $z$ marking. Similarly, from
  each $w$ marking, as one travel along $\del S$ to the next $z$
  marking,  isotope all the $\be$ circles that one encounters, by
  finger moves, across the $z$ marking. 
\item Finally, perform  isotopies of the $\alpha$ circles and the
  $\beta$ circles in $\Si\sm(z\cup w)$ to make the diagram admissible.
\end{enumerate}
\end{alg}

\noindent Note that such a Heegaard diagram is indeed 
adapted to $R$. In particular, (\ref{item:alpha-max-intersect})
ensures that each generator has at most $(k-\chi(R))$ coordinates
inside $R$. 

We now  spell out an explicit way of making all the choices alluded to in the previous list. The process is best understood in conjunction with an explicit example, as illustrated in Figure
\ref{fig:example}. At various points it will be useful to make minor alterations to
these choices, but for the sake of brevity (and sanity), we will not explicitly describe all the  choices made each time a Heegaard diagram is constructed.

\begin{figure}
\centering
\includegraphics[scale=0.7]{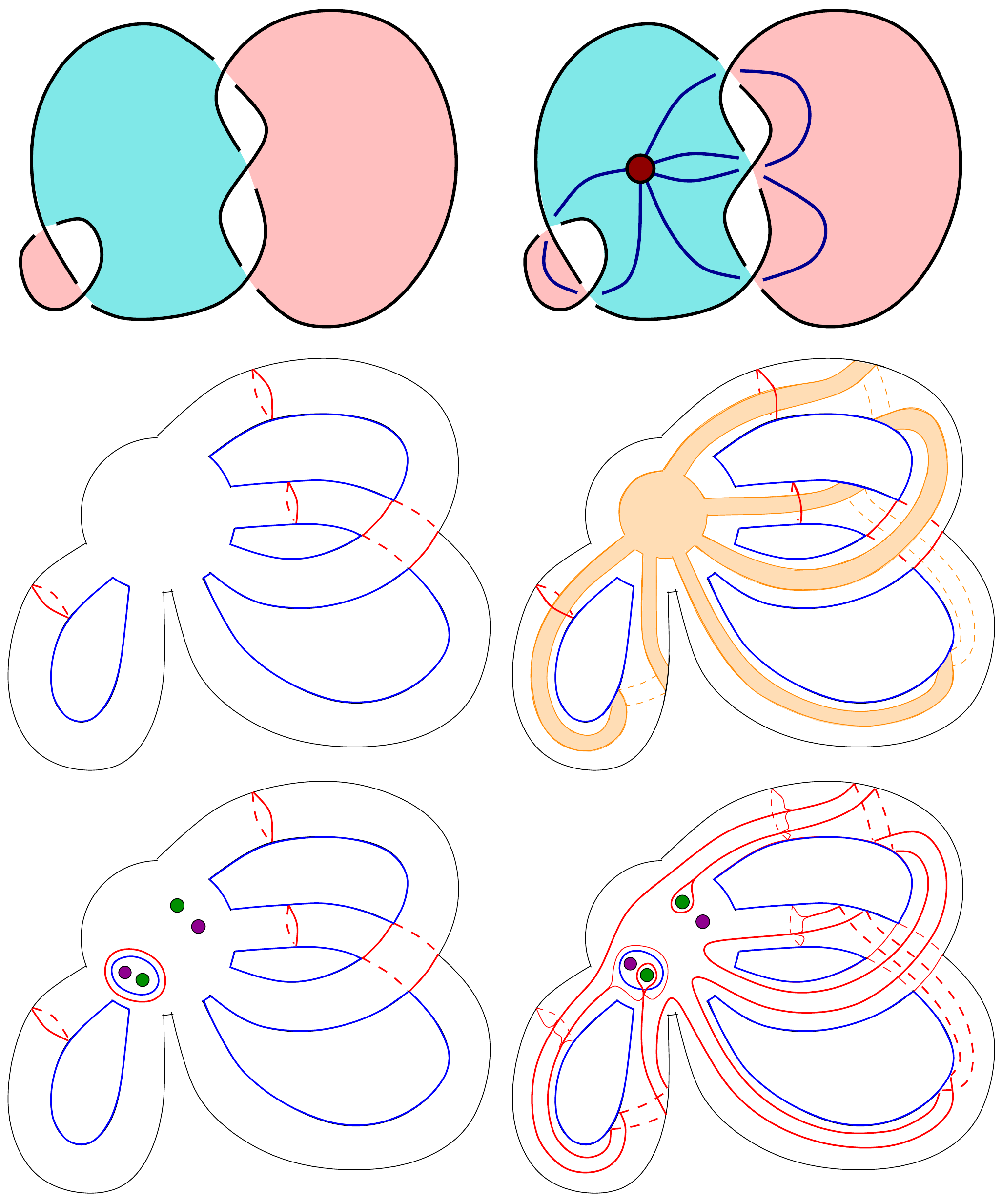}
\caption{\textbf{An algorithm for constructing a Heegaard diagram
  adapted to a Seifert surface.} As usual, the red circles are $\alpha$
  and the blue ones are $\beta$. The surface $S$ is orange. The
  magenta dots are $w$-markings and the green dots are
  $z$-markings. In the last diagram, the $\al$ circles are represented
  by train tracks, with the thin red lines denoting
  curves.}\label{fig:example}
\end{figure}

\begin{alg} {\em Explicit diagram adapted to a planar projection of an embedded Seifert surface} $R\subset S^3$.

\begin{enumerate}[leftmargin=*,label=(E-\arabic*),ref=E-\arabic*]
\item Given a Seifert surface $R$ for an $l$-component link
  $L\sbs S^3$, view it as a surface lying in $\R^3$. Consider a
  projection $\pi:\R^3\rightarrow\R^2$, and assume that $\pi|_R$ is
  generic and the image $\pi(R)$ is connected.
\item If $R$ has $n$ components, let $G$ be a graph with $n$ vertices
  and $(n-\chi(R))$ edges, embedded in the interior of $R$, such that
  the following holds:
  \begin{enumerate}[leftmargin=*]
  \item $R$ deform retracts to $G$.
  \item The vertex $v$ is a regular point of $\pi|_R$.
  \item $\pi|_G$ is an immersion with no triple points, and all the
    preimages of the double points lie in the interior of the
    edges.
  \end{enumerate}
\item Let $U_{\al}=\ol{\nbd_{S^3}(\pi(G))}$ be a genus $g$ handlebody
  and let $U_{\be}=S^3\sm \nbd_{S^3}(\pi(G))$ be the complementary
  handlebody. Let $\Si=\del U_{\al}$ be the Heegaard surface.
\item Designate $g$ of the $(g+1)$ circles in $\Si\cap\R^2$ as $\be$
  circles. 
\item For each of the $(n-\chi(R))$ edges of $G$, choose a point in
  the image of the interior of the edge that is not a double point,
  and draw an $\alpha$ circle on $\Sigma$ which is the boundary of a normal disk to $G$ inside
  $U_{\al}=\ol{\nbd_{S^3}(\pi(G))}$. 
  Draw an additional  $\al$ circle near each of the $(g+\chi(R)-1)$
  double points of $\pi|_G$, such that the $\al$ circle bounds a disk
  in $U_{\al}$ near the double point, and if the disk were surgered
  out, then $U_{\al}$ locally would have two components, corresponding
  to the two preimages of the double point, with the same crossing
  information. 
\item For each vertex $v_i$ of $G$, let $p_i\in\Si$ be the unique
  point such that $\pi(p_i)=\pi(v_i)$ and $|d\pi|_{_\Si}(p_i)|$  has
  the same sign as $|d\pi|_{_R}(v_i)|$. Let $D_i\sbs \Si$ be a small
  disk containing $p_i$; let $D=\cup_i D_i$.
\item\label{item:attach-bands-carefully} For each of the $(n-\chi(R))$ edges of $G$, attach a band to $D$  lying in
  $\Si\sm (\al\text{ circles near the double
    points})$. Choose each band so that it
  deformation retracts onto an arc that projects to the corresponding
  edge under $\pi$, and so that the surface framing of the band in $\Si$ is
  same as the surface framing of the corresponding edge in $R$. Let
  $S\subset \Si$ be the surface obtained from $D$ by adding the bands. 
\item Put $2l$
  markings $z=(z_1,\ldots,z_l)$ and $w=(w_1,\dots,w_l)$ on
  $\del S\cap\del D$, such that each component of $\del S$ contains
  exactly one $z$ marking and exactly one $w$ marking, and the $l$
  arcs $b_1,\cdots,b_l\sbs\del S$, which join the $w$ markings to the
  $z$ markings, are supported inside $\del S\cap\del D$.
\item For each disk $D_i$, add an $\alpha$ circle around all but one
  of the $b_j$'s supported in $D_i$. This adds a total of $(l-g)$
  $\al$ circles.
\item For $2\leq i\leq l$, add a $\be$ circles around $b_i$.
\item Perform finger moves on the $\al$ circles, as
  described  in~(\ref{item:finger-moves-heegaard-diagram}), to
  obtain the final Heegaard diagram. One can check
  that the diagram thus obtained is admissible. Furthermore,
  since the surface $S$ was disjoint from the $(g+\chi(R)-1)$ $\alpha$
  circles near the double points,
  see~(\ref{item:attach-bands-carefully}), it remains disjoint from them
  even after the finger moves, and consequently, it only intersects
  $(g+l-1)-(g+\chi(R)-1)=(l-\chi(R))$ $\alpha$ circles.
\end{enumerate}
\end{alg}

\subsection{Knot Floer homology}\label{subsec:knotfloer}
We  briefly recall the definition of the ``tilde" version of
Heegaard Floer homology, essentially following
\cite[Section 6.1]{POZSzlinkinvariants} cf. \cite[Proposition 2.5]{MOS}. Given a Heegaard diagram for $S^3$, 
$\mc{H}=(\Sigma_{(g)},\alpha^{(g+k-1)},\beta^{(g+k-1)},w^{(k)})$, the chain complex $\tCF_{\mc{H}}$ is the
$\F_2$-module freely generated by the elements of $\mc{G}_{\mc{H}}$.

Given generators $x,y\in\mc{G}_{\mc{H}}$, a
\emph{domain} joining them is a $2$-chain $D$ generated by the
elementary regions of $\mc{H}$ such that $\del(\del D\cap\al)=y-x$;
here, an \emph{elementary region} is the closure of a component of
$\Si\sm(\al\cup\be)$, and we are thinking of the generators as formal
linear sums of their coordinates.  The set of all the domains joining
$x$ to $y$ is denoted by $\mc{D}(x,y)$. A domain $D$ is said to be
\emph{positive} if all its coefficients are non-negative, and at least
one of the coefficients is positive. Given a point
$p\in\Si\sm(\al\cup\be)$, let $n_p(D)$ denote the coefficient of $D$
at the elementary region containing the point $p$; 
let
$n_w(D)=\sum_{i=1}^{k} n_{w_i}(D)$.
Domains with $n_w(D)=0$ are
called \emph{empty domains}, and the set of all empty domains joining
$x$ to $y$ is denoted by $\mc{D}_0(x,y)$. 
Elements of $\mc{G}_{\mc{H}}$ carry a well-defined grading called the
\emph{absolute Maslov grading} $M$, which serves as the homological
grading of $\tCF_{\mc{H}}$. The difference in Maslov gradings can be computed as
\[M(x)-M(y)=\mu(D)-2n_w(D),\] where $D\in\mc{D}(x,y)$ is any domain, and $\mu(D)$ denotes its \emph{Maslov index}.

After choosing a generic path of almost complex structures on
$\mathrm{Sym}^{g+k-1}(\Si)$, sufficiently close to the constant path of one
induced from a complex structure on $\Si$, one can define the
\emph{contribution function} $c$, from the set of all empty Maslov
index one domains, to $\F_2$, given by $c(D)=|\mc{M}(D)/\R|$, the
number of points in a certain unparametrized moduli space. The
function $c$ has the property that it evaluates to $1$ only if
\begin{enumerate}[leftmargin=*,label=(\alph*)]
\item the domain is positive \cite[Lemma 3.2]{POZSz}, and
\item the closure of the union of the elementary regions where the
  domain is supported is connected \cite[Corollary 9.1]{JR}. 
\end{enumerate}
Then the boundary map on the chain complex $\tCF_{\mc{H}}$ is given by
\[\del x=\sum_{y\in\mc{G}_{\mc{H}}}
  \sum_{\substack{D\in\mc{D}_0(x,y)\\\mu(D)=1}}c(D)y.\] 
(The chain complex $\tCF_{\mc{H}}$
usually depends on the choice of the path of almost complex structures
on $\mathrm{Sym}^{g+k-1}(\Si)$; nevertheless, we will suppress this
from the notation.)

\begin{thm}\cite{POZSzlinkinvariants, MOS}\label{thm:main-invariance}
  The homology $\tHF_{\mc{H}}$ of the chain complex $\tCF_{\mc{H}}$
  coming from a Heegaard diagram
  $\mc{H}=(\Sigma_{(g)},\al^{(g+k-1)},\be^{(g+k-1)},w^{(k)})$ for
  $S^3$ is isomorphic, as graded $\F_2$-modules, to
  $\otimes^{k-1}(\F_2\oplus\F_2[-1])$, where $[i]$ denotes a grading
  shift by $i$.
\end{thm}

The tilde version of knot Floer homology or link Floer homology
\cite{POZSzknotinvariants, JR, POZSzlinkinvariants} is a refinement of
Heegaard Floer homology. Let
$\mc{H}=(\Sigma_{(g)},\alpha^{(g+k-1)},\beta^{(g+k-1)},z^{(k)},w^{(k)},S)$
be a Heegaard diagram adapted to a Seifert surface $R$ of an
$l$-component link $L\sbs S^3$. Consider the Heegaard diagram
$\mc{H}_0=(\Sigma_{(g)},\alpha^{(g+k-1)},\beta^{(g+k-1)},w^{(k)})$
obtained by forgetting $S$ and the $z$ markings. The set of generators
$\mc{G}_{\mc{H}}$ is same as $\mc{G}_{\mc{H}_0}$, and they carry the
same absolute Maslov grading. Given a $2$-chain $D$
generated by the elementary regions of $\mc{H}$, let
$n_z(D)=\sum_{i=0}^kn_{z_i}(D)$. The elements of $\mc{G}_{\mc{H}}$
carry another well-defined grading called the \emph{absolute Alexander
  grading} $A$, such that for any domain $D\in\mc{D}(x,y)$,
$A(x)-A(y)=n_z(D)-n_w(D)$.

\begin{prop}\label{prop:alex-grading-range}
If $x\in\mc{G}_{\mc{H}}$ is a generator in
$\mc{H}=(\Sigma_{(g)},\alpha^{(g+k-1)},\beta^{(g+k-1)},z^{(k)},w^{(k)},S)$,
then its absolute Alexander grading is given by $A(x)=(\text{number of
}x\text{-coordinates inside }S)-\frac{1}{2}(2k-l-\chi(R))$. In
particular, the Alexander grading satisfies:
$-\frac{1}{2}(2k-l-\chi(R))\leq A(x)\leq\frac{1}{2}(l-\chi(R))$. 
\end{prop}

\begin{proof} 
 The Alexander grading
  of a generator $x\in\mc{G}_{\mc{H}}$ is given by
  $\frac{1}{2}\langle c_1(s(x)),[R,\partial R]\rangle$. If the
  generator has no coordinate inside $S$ (called outer in \cite{AJ}),
  we can evaluate it as $\frac{1}{2}c(S)$ where $c(S)$ is the quantity
  defined in \cite[Section~3]{AJ}. Then using \cite[Lemma~3.9]{AJ}, we
  see that
  \[
    c(S)=\chi(S)+I(S)-r(S)=\chi(S)-k-(k-l)=\chi(R)+l-2k,
  \]
  where $I(S)$ equals minus half the number of sutures (which are
  basepoints in our setting) and $r(S)$ is $0$ if there is exactly one
  pair of sutures on each boundary and decreases by one for each
  additional pair of sutures.

  For generators which are not disjoint from $S$, we only need to
  notice that $c_1(s(x))-c_1(s(y))=2\mathrm{PD}(\alpha)$, where
  $\alpha=\del D$ for any domain $D\in\mc{D}(x,y)$. It is not hard to
  see that the algebraic intersection number of $\alpha$ with $\del S$
  equals the number of $x$-coordinates inside $S$ minus the number of
  $y$-coordinates inside $S$, and therefore,
  \[
    A(x)=\frac{1}{2}(\chi(R)+l-2k)+(\text{number of
    }x\text{-coordinates inside }S).
  \] 
  This proves the lefthand side of
  the inequality. For the righthand side, we only need to use the
  following fact
  \[
    (\text{number of }x\text{-coordinates inside }S)\le(\text{number
      of }\alpha\text{ circles intersecting }S)= k-\chi(R).\qedhere
  \]
\end{proof}

In view of the above proposition, we make the following definitions.
\begin{defn}
Given a compact surface $R$ (possibly disconnected), define its {\em index} to be
$\frak{i}(R)=\frac{1}{2}(|\partial R|-\chi(R))$.  Call a Seifert surface $R$ for a
link $L$ {\em minimal} if it minimizes the index, and define the {\em genus} of
the link, $g(L)$, to be this minimal index.
\end{defn}

It is easy to see that the chain complex
$\tCF_{\mc{H}}=\tCF_{\mc{H}_0}$ is filtered by the Alexander
grading. Let the filtration level
$\mc{F}_{\mc{H}}(m)\subseteq \tCF_{\mc{H}}$ denote the subcomplex
generated by the generators with Alexander grading $m$ or less. We
call such an $(M,A)$-bigraded complex, where the differential
decreases $M$ by one and does not increase $A$, to be an
$M$-graded-$A$-filtered complex.

\begin{thm}\cite{POZSzknotinvariants,JR,POZSzlinkinvariants}\label{thm:link-invariance}
  To an $l$-component link $L\subset S^3$, one can associate (the
  filtered chain homotopy type of) an $M$-graded-$A$-filtered complex
  $\CFK(L)$ such that the chain complex $\tCF_{\mc{H}}$, coming from
  any Heegaard diagram
  $\mc{H}=(\Sigma_{(g)},\al^{(g+k-1)},\be^{(g+k-1)},z^{(k)},w^{(k)},S)$
  adapted to any Seifert surface $R$ for $L$, is filtered
  chain homotopy equivalent to
  $\CFK(L)\otimes^{k-l}(\F_2\oplus\F_2[-1,-1])$, where $[i,j]$ denotes
  the $(M,A)$ bi-grading shift by $(i,j)$.
\end{thm} 

It is clear from Proposition~\ref{prop:alex-grading-range} and
Theorem~\ref{thm:link-invariance} that the subcomplex of $\CFK(L)$ in
Alexander grading less than $-\frak{i}(R)$ is filtered
chain homotopy equivalent to zero. Let
$\hatCFK(L,-\frak{i}(R))$ denote the subcomplex of
$\CFK(L)$ in Alexander grading less than or equal to
$-\frak{i}(R)$. If $R$ is minimal, then its homology,
$\hatHFK(L,-g(L))$, is
non-zero~\cite{POZSzgenusbounds,YNlinkgenusbounds} carrying a single
grading coming from the Maslov grading, and is
called the \emph{extremal knot Floer homology}. 

Instead of studying the full filtration on $\CFK(L)$, we will 
restrict our attention to the  two-step filtration 
$\hatCFK(L,-\frak{i}(R)) \subset \CFK(L)$. Recall that a
\emph{two-step filtered complex} is simply a pair
$(S,C)$ where $C$ is a chain complex and $S\subset C$ is a subcomplex. A
filtered chain map $f$ from $(S,C)$ to $(S',C')$ is a chain map
$f\colon C\to C'$ so that $f(S)\subseteq S'$. A filtered chain map
$f$ from $(S,C)$ to $(S',C')$ is a \emph{quasi-isomorphism} if both
$f\colon C\to C'$ and $f|_S\colon S\to S'$ induce isomorphisms on
homology.  We will make use of the following corollary of
Theorem~\ref{thm:link-invariance}

\begin{cor}\label{cor:2-step-filtered}
  Let
  $\mc{H}=(\Sigma_{(g)},\al^{(g+k-1)},\be^{(g+k-1)},z^{(k)},w^{(k)},S)$
  be a Heegaard diagram adapted to a minimal Seifert surface $R$ for
  $L$. Then there is a quasi-isomorphism of pairs 
  \[
  (\mc{F}_{\mc{H}}(-\frak{i}(R)-k+l),\tCF_{\mc{H}})\cong
  (\hatCFK(L,-g(L))\otimes^{k-l}(\F_2[-1,-1]),\CFK(L)\otimes^{k-l}(\F_2\oplus\F_2[-1,-1])).
  \]
  In
  particular, the extremal knot Floer homology is isomorphic to
  the homology of
  $\mc{F}_{\mc{H}}(-\frak{i}(R)-k+l)[k-l]$. Moreover, the
  maps on homologies, $\hatHFK(L,-g(L))\to H_*(\CFK(L))$
  and
  $H_*(\mc{F}_{\mc{H}}(-\frak{i}(R)-k+l))\rightarrow
  \tHF_{\mc{H}}$ have the same rank.
\end{cor}

We conclude this section by describing how the extremal knot Floer
homology is related to the $\tau$-invariant. Ozav\'ath-Szab\'o
originally defined the $\tau$-invariant for knots in $S^3$; there are
a number of generalizations of this invariant to links, and we will concentrate on $\taubot$ and $\tautop$ which, by \cite[Proposition~5.16]{HeddenRaoux} correspond to the smallest and largest of all the possible $\tau$ invariants for links (For a knot $K$,
$\tau(K)=\taubot(K)=\tautop(K)$.) For now, we only need the following
properties of these invariants.
\begin{prop}\label{prop:tau-std-prop}
If $m(L)$ denotes the mirror of $L$, then $\taubot(m(L))=-\tautop(L)$.
The invariant $\taubot$ satisfies $-g(L)\leq\taubot(L)\leq g(L)$ with
$\taubot(L)=-g(L)$ if and only if the map $\hatHFK(L,-g(L))\to
H_*(\CFK(L))$ is non-zero.  
\end{prop}

\begin{proof}
 The relationship between $\taubot$ and $\tautop$ under mirroring follows from their definition, a  duality property satisfied by generalized $\tau$ invariants  \cite[Proposition~2.5]{HeddenRaoux}.  That $\taubot$ is bounded by the genus of $L$ follows from the fact that it is correspondingly bounded by the ``slice genus" \cite[Proposition~5.14]{HeddenRaoux}.  The final statement is a consequence of the definition of $\taubot$, and the monotonicity of the $\tau$ invariants for links established in \cite[Proposition~5.16]{HeddenRaoux}.  See \cite[Theorem~2 and Section~5.3]{HeddenRaoux} for more details.
\end{proof}

\begin{prop}
  $\mc{H}=(\Sigma_{(g)},\al^{(g+k-1)},\be^{(g+k-1)},z^{(k)},w^{(k)},S)$
  be a Heegaard diagram adapted to a minimal Seifert surface $R$ for a
  knot $L\sbs S^3$. Then $\taubot(L)=-g(L)$ 
  if and only if the map on homology
  $H_*(\mc{F}_{\mc{H}}(-g(L)-k+l))\rightarrow \tHF_{\mc{H}}$ induced
  from the inclusion
  $\mc{F}_{\mc{H}}(-g(L)-k+l)\hookrightarrow\tCF_{\mc{H}}$ is
  non-trivial.
  \end{prop}

\begin{proof}
This follows immediately from Proposition~\ref{prop:tau-std-prop} and Corollary~\ref{cor:2-step-filtered}.
\end{proof}

\subsection{Triangle maps}\label{sec:triangle-map}

We briefly introduce the definition of triangle maps in our
restricted setting, once again following the original definitions from
\cite{POZSz}. Let
$\mc{H}=(\Si_{(g)},\al^{(g+k-1)},\be^{(g+k-1)},\ga^{(g+k-1)},
w^{(k)})$ be a \emph{triple Heegaard diagram}, that is:
$\mc{H}_{\al\be}=(\Si,\al,\be,w)$ and
$\mc{H}_{\ga\be}=(\Si,\ga,\be,w)$ are Heegaard diagrams for $S^3$;
$\al_i$ is disjoint from $\ga_j$ for $i\neq j$; $\al_i$ is transverse
to $\ga_i$ and they intersect each other in exactly two points, none
of which lies on the $\be$ curves; furthermore, if $\alpha_j$ bounds a
disk $D_j$ in the $\alpha$-handlebody $U_\al$, then $\ga_i$ is
isotopic to $\al_i$ in
$\nbd_{U_\al}(\bigcup_{j\neq i}D_j)\cup(\Si\setminus w)$, that is,
$\ga_i$ can be isotoped to $\al_i$ after sliding it over some other
$\al$ circles in the complement of the $w$ markings. We will once
again assume that the triple Heegaard diagram is \emph{admissible}.

Orient $\al_i$ arbitrarily, and then orient $\ga_i$ in the same
direction, induced from the isotopy joining $\ga_i$ to $\al_i$. Let
$\theta_i$ be the positive intersection point in $\ga_i\cap\al_i$, and
let $\theta=(\theta_1,\ldots,\theta_{g+k-1})$. It is usually called
the \emph{top generator}. Note  that
$(\Sigma,\gamma,\alpha,w)$ is a Heegaard diagram for
$\#^{g}(S^1\times S^2)$ on which $(k-1)$ index 0/3 stabilizations have been performed, and $\theta$ is its unique generator of highest Maslov grading.

\emph{Elementary regions} of $\mc{H}$ are closures of the components
of $\Si\sm(\al\cup\be\cup\ga)$; 
a \emph{triangular domain} joining a generator
$x\in\mc{G}_{\mc{H}_{\al\be}}$ to a generator
$y\in\mc{G}_{\mc{H}_{\ga\be}}$ is a $2$-chain $D$ generated by the
elementary regions such that $\del(\del D\cap{\al})=\theta-x$ and
$\del(\del D\cap{\be})=x-y$; a triangular domain is said to be
\emph{positive} if all its coefficients are non-negative. Given a
triangular domain $D$, let $n_w(D)=\sum_i n_{w_i}(D)$, where
$n_{w_i}(D)$ is the coefficient of the elementary region containing
$w_i$, in the $2$-chain $D$. Let $\mc{T}(x,y)$ be the set of all
triangular domains joining $x\in\mc{G}_{\mc{H}_{\al\be}}$ to
$y\in\mc{G}_{\mc{H}_{\ga\be}}$, and let $\mc{T}_0(x,y)$ be the subset
consisting of the \emph{empty triangular domains}, that is, triangular
domains with $n_w=0$. The \emph{Maslov grading} $\mu(D)$ of any
triangular domain $D\in\mc{T}(x,y)$ satisfies
$\mu(D)-2n_w(D)=M(y)-M(x)$.

Choosing an appropriate family (parametrized by the 2-simplex) of almost complex structures on $\mathrm{Sym}^{g+k-1}(\Si)$, we can define a \emph{contribution
  function} $c$ from the set of all Maslov index zero triangular
domains, to $\F_2$.  Picking the family of almost complex structures to be integrable near a collection of hypersurfaces specified by basepoints in the elementary regions ensures that the contribution function has non-zero support only on  the positive
triangular domains. Then the following map is a graded
quasi-isomorphism from $\tCF_{\mc{H}_{\al\be}}$ to $\tCF_{\mc{H}_{\ga\be}}$.
\[
f(x)=\sum_{y\in\mc{G}_{\mc{H}_{\ga\be}}}
\sum_{\substack{D\in\mc{T}_0(x,y)\\\mu(D)=0}}c(D)y.
\]
We will reprove a special case of this fact in
Theorem \ref{thm:simplify}.

\subsection{Murasugi sum}\label{subsec:murasugi}

We are now prepared to discuss the Murasugi sum operation. Let
$S^2\sbs S^3$ be the standard $2$-sphere. We will mentally
`one-point-decompactify' the picture, and draw it as
$\R^2\sbs\R^3$. There are two components in $S^3\sm S^2$, the `inside'
$B_1$ and the `outside' $B_2$, such that $S^2$ is oriented as $\del
B_1$. Let $A_1A_2\dots A_{2n}$ be a $2n$-gon lying on $S^2$. For
$i\in\{1,2\}$, let $R_i$ be a Seifert surface for an
$l_i$-component link $L_i\sbs\ol{B_i}$, such that: $R_i\cap S^2$ is
the $A_1A_2\dots A_{2n}$ with the same orientation; $L_1\cap S^2$ is
the union of the oriented segments $A_1A_2$, $A_3A_4$, \dots,
$A_{2n-1}A_{2n}$; and $L_2\cap S^2$ is the union of the oriented
segments $A_2A_3$, $A_4A_5$, \dots, $A_{2n}A_1$. Then the \emph{Murasugi
  sum} $L=L_1*L_2$ is the link $\ol{(L_1\cup L_2)\sm S^2}$, and it
bounds the Seifert surface $R_1*R_2=R_1\cup R_2$. The special cases when $n=1$ is just the connected sum and $n=2$ is a plumbing. The case $n=2$ is
illustrated in Figure \ref{fig:bothlinks}.

\begin{figure}
\centering
\includegraphics{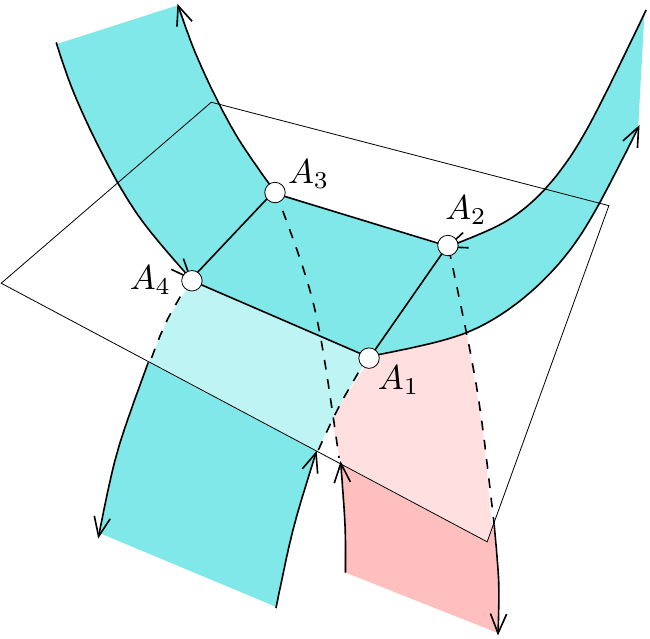}
\caption{\textbf{The Murasugi sum operation.} The link $L$ is obtained
  by plumbing the link $L_1$ below the plane with the link $L_2$ above
  the plane along the rectangle $A_1A_2A_3A_4$.}\label{fig:bothlinks}
\end{figure}

We need the following quantities, $\delta_1,\delta_2,\delta$, which
will simplify certain expressions later on.  Define $\delta_1$
(respectively, $\delta_2,\delta$) to be $n$ minus the number of
components of $L_1$ (respectively, $L_2$, $L$) that intersect the
$2n$-gon. Note, $(l+\delta-n)=(l_1+\delta_1-n)+(l_2+\delta_2-n)$.

We will now describe how to draw Heegaard diagrams adapted to $R_1$
and $R_2$, and how they can be combined to form a Heegaard diagram for
$R_1*R_2$. The procedures closely follow the outline from
Section~\ref{subsec:intro} with a few differences. The most important
feature of the following construction is that the roles of $\al$ and
$\be$ are reversed while constructing the Heegaard diagrams adapted to
$R_1$ and $R_2$.

\begin{figure}
\centering
\includegraphics{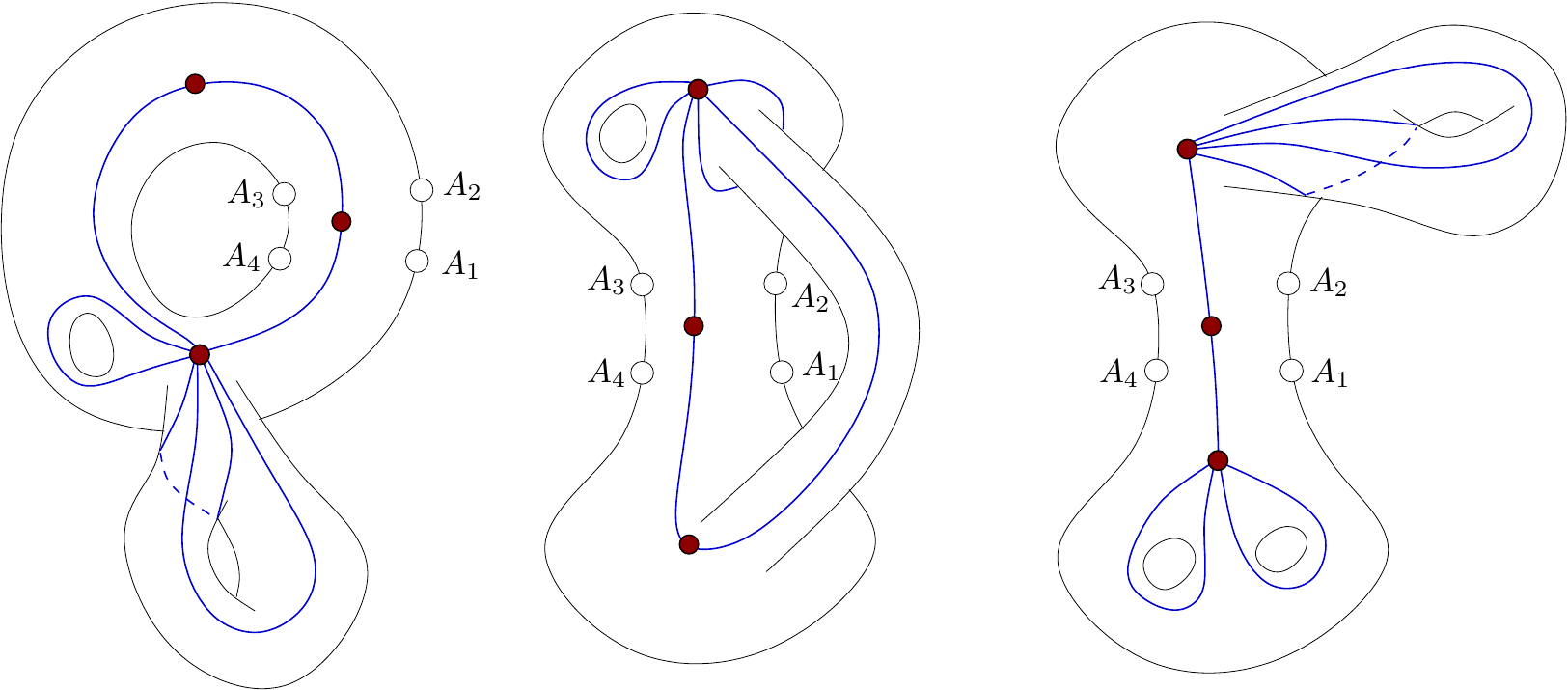}
\caption{\textbf{Embedding a graph $G$ in the Seifert surface.} In
  each case, the Seifert surface deform retracts to a neighborhood of
  $G$ that contains $A_1A_2A_3A_4$.}\label{fig:caseanalysis}
\end{figure}

\begin{enumerate}[leftmargin=*,label=(M-\arabic*),ref=M-\arabic*]
\item We first embed a graph $G_i$ in $R_i$ so that $R_i$ deform
  retracts to $G_i$, and $G_i$ intersects the $2n$-gon in a single
  vertex with exactly $n$ edges going out to $n$ of its edges. This is
  easy to ensure, see Figure~\ref{fig:caseanalysis}.
\item Then consider the handlebody $\ol{B_{3-i}\cup
    \nbd_{B_i}(G_i\cap B_i)}$. Its complement need not be a handlebody, so we
  add a few tunnels, none intersecting $S^2$, to complete this to a
  Heegaard decomposition of $S^3$. Let $\Si_i$ be the resulting
  Heegaard surface, oriented so that its orientation agrees with the
  orientation of $S^2=\del B_1$ on $S^2\cap \Si_i$. Let $U_{\al,i}$
  and $U_{\be,i}$ be the components of $S^3\sm\Si_i$, so that $\Si_i$
  is oriented as the boundary of $U_{\al,i}$. 
\item Construct a surface $S_i\sbs\Si_i$, such that $S_i$ is isotopic
  to $R_i$ and $S_i\cap S^2$ is the $2n$-gon  $A_1A_2\dots A_{2n}\cap\Si_i$.
\item Put $z$ markings at $A_1,A_3,\dots,A_{2n-1}$, and put $w$
  markings at $A_2,A_4,\dots,A_{2n}$. On every other component of
  $\del S_i$, put a $z$ marking and a $w$ marking right next to one
  another, such that a small arc in $(-1)^i\del S_i$ joins the $w$
  marking to the $z$ marking. Therefore, the total number of $z$ (or
  $w$) markings is $l_i+\delta_i$.
\item Then draw $\al$ circles and $\be$ circles on $\Si_i\sm(z\cup
  w)$, such that:
  \begin{enumerate}[leftmargin=*]
  \item The $\al$ circles and $\be$ circles are transverse to each
    other and to $\del S_i$.
  \item The $\al$ circles are pairwise disjoint and they span a
    half-dimensional subspace of $H_1(\Si_i)$; each component of
    $\Si_i\sm\al$ contains a $z$ marking and a $w$ marking.
  \item The $\be$ circles are pairwise disjoint and they span a
    half-dimensional subspace of $H_1(\Si_i)$; each component of
    $\Si_i\sm\be$ contains a $z$ marking and a $w$ marking.
  \item Each component of $\Si_1\sm\al$ (respectively, $\Si_2\sm\be$)
    has an oriented arc in $(-1)^i\del S_i$ joining the $w$ marking to
    the $z$ marking.
  \item Exactly $(l_i+\delta_i-\chi(R_i))$ $\be$ (respectively, $\al$)
    circles intersect $S_1$ (respectively, $S_2$).
  \item There are exactly $(n-1)$ $\al$ circles lying entirely inside
    the $2$-sphere $S^2$, and they encircle the edges $A_3A_4$,
    $A_4A_5$, \dots, $A_{2n-1}A_{2n}$. There are exactly $(n-1)$ $\be$
    circles lying entirely inside the $2$-sphere $S^2$, and they
    encircle the intervals $A_4A_5$, $A_6A_7$, \dots,
    $A_{2n}A_1$. Moreover, the consecutive $\alpha$ and $\beta$
    circles intesect each other at exactly two points.
  \item Other than the above circles, there are no $\be$
    (respectively, $\al$) circle of $\Si_1$ (respectively, $\Si_2$)
    that intersects $S^2$. There could be some $\al$ (respectively,
    $\be$) arcs of $\Si_1$ (respectively, $\Si_2$) that intersect
    $S^2$; in that case, their intersection with $S^2$ lies entirely
    inside $S_i\cap S^2$; and we can also ensure that there are at
    most $(n-1)$ of such $\al$ (respectively, $\be$) arcs.
  \end{enumerate}
\item We then do finger moves on the $\be$ (respectively, $\al$)
  circles on $\Si_1$ (respectively, $\Si_2$) to convert this to a
  Heegaard diagram $\mc{H}_1$ (respectively, $\mc{H}_2$) adapted to
  the Seifert surface $R_1$ (respectively, $R_2$). These final
  diagrams, in the case when $n=2$, are shown
  in Figure~\ref{fig:bothlinksdiagram}. In the case when $n=3$, the
  diagrams are also shown in the first two figures of the top row of
  Figure~\ref{fig:modified-diagrams} (where X denotes a handle going
  down into $B_1$ and O denotes a handle coming up into $B_2$).
\item\label{item:make-stuff-admissible} The first two figures in the third row of
  Figure~\ref{fig:modified-diagrams} represent slightly modified
  Heegaard diagrams $\mc{H}''_i$ that are obtained from $\mc{H}_i$ by
  deleting the $z$-markings, modifying the surface $S_i$, and
  performing small isotopies to reduce the number of intersections
  between $\al$ and $\be$ circles. We will assume that we have already
  performed some isotopies on the $\al$ and $\be$ circles on
  $\mc{H}_i$ away from $S^2$ so as to ensure that these modified
  diagrams $\mc{H}''_i$, and hence $\mc{H}_i$ itself, is already
  admissible.
\end{enumerate}

\begin{figure}
\centering
\includegraphics[scale=0.6]{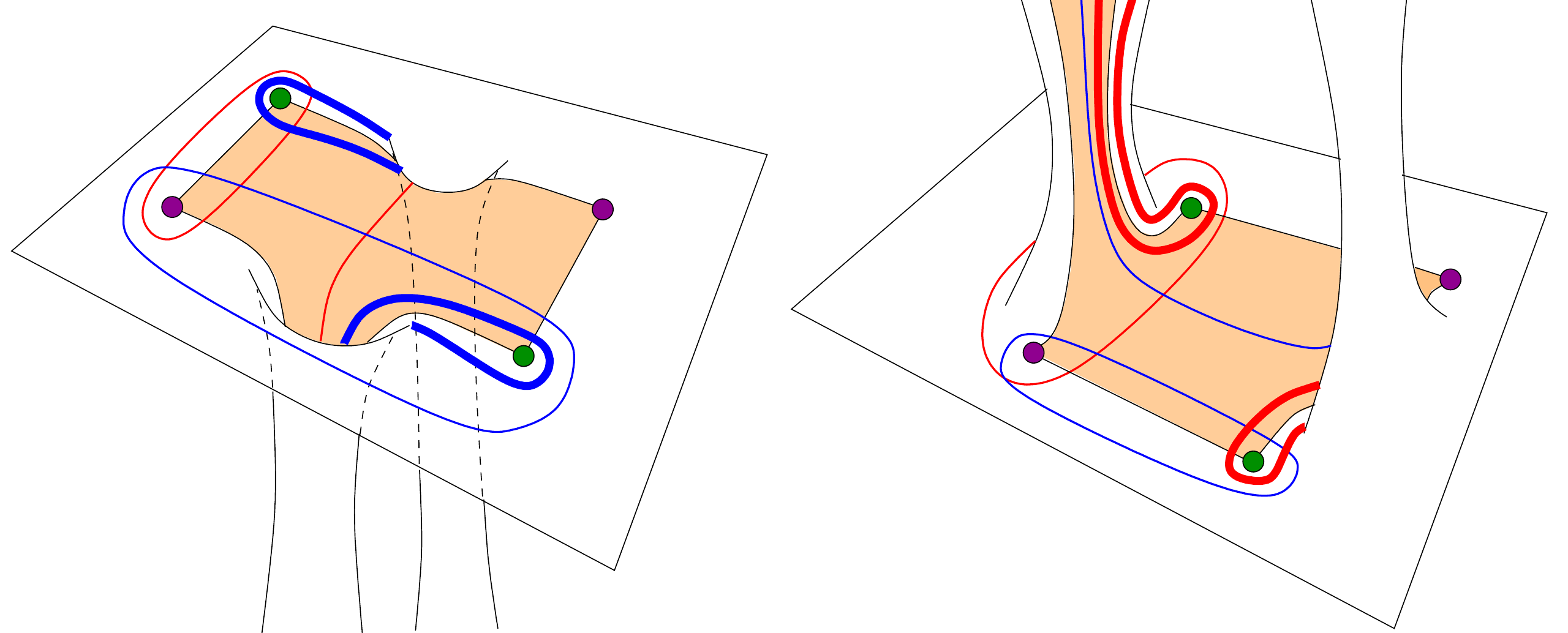}
\caption{\textbf{The Heegaard diagrams $\mc{H}_1$ and $\mc{H}_2$.} We
  continue to represent $w$ and $z$ markings by magenta and green
  dots, respectively. Once again, the thin lines are curves, and the
  thick lines are train tracks, representing some (possibly zero)
  curves running in parallel.}\label{fig:bothlinksdiagram}
\end{figure}

\begin{figure}
\centering
\includegraphics[scale=0.6]{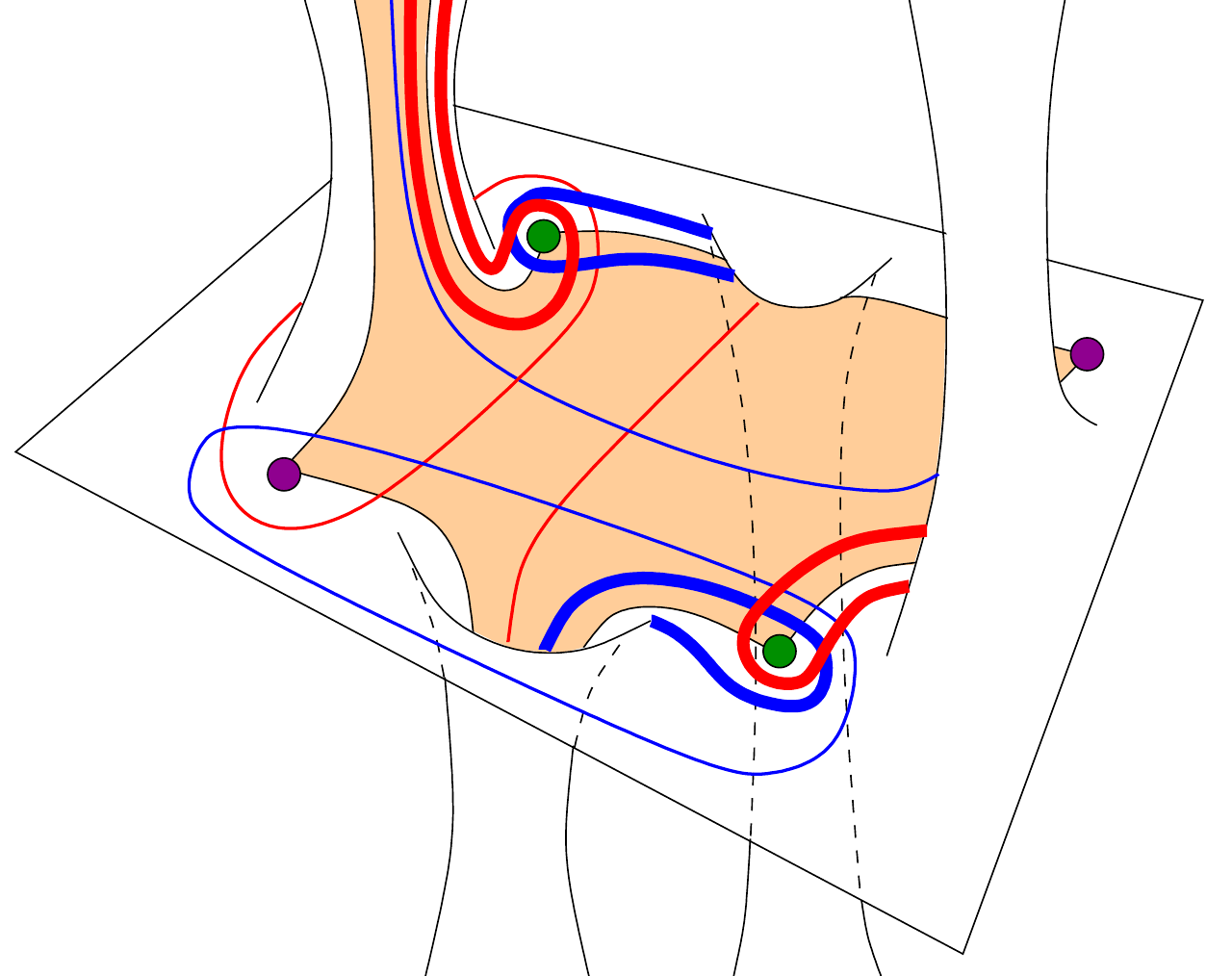}
\caption{\textbf{The Heegaard diagram $\mc{H}_1*\mc{H}_2$.} This diagram is obtained by combining the Heegaard diagrams $\mc{H}_1$ and $\mc{H}_2$ from Figure~\ref{fig:bothlinksdiagram}}\label{fig:combinedheegaard}
\end{figure}

We can now `combine' the Heegaard diagram $\mc{H}_1$ adapted to $R_1$
and $\mc{H}_2$ adapted to $R_2$ to form a Heegaard diagram
$\mc{H}_1*\mc{H}_2$ adapted to $R_1*R_2$. Recall that in $\mc{H}_1$,
the $\al$-handlebody $U_{\al,1}$ is obtained by tunneling out a few
one-handles from $\ol{B_1}$, and the $\be$-handlebody $U_{\be,1}$ is
obtained by attaching those corresponding one-handles to
$\ol{B}_2$. Similarly, in $\mc{H}_2$, the $\al$-handlebody $U_{\al,2}$
is obtained by attaching a few one-handles to $\ol{B_1}$, and the
$\be$-handlebody $U_{\be,2}$ is obtained by tunneling out those
corresponding one-handles from $\ol{B}_2$. In the `combined' Heegaard
diagram $\mc{H}_1*\mc{H}_2$, the $\al$-handlebody
$U_{\al,1}*U_{\al,2}$ is obtained from $\ol{B}_1$ by tunneling out all
the one-handles that were tunneled out in $U_{\al,1}$ and by attaching
all the one-handles that were attached in $U_{\al,2}$, and the
$\be$-handlebody $U_{\be,1}*U_{\be,2}$ is the closure of its
complement. The Heegaard surface $\Si_1*\Si_2$ is the oriented
boundary of $U_{\al,1}*U_{\al,2}$. There is a surface
$S_1*S_2\sbs\Si_1*\Si_2$, isotopic to $R_1*R_2$, which is obtained
from $S_1\sbs\Si_1$ and $S_2\sbs\Si_2$. The $w$ and $z$ basepoints,
and the $\al$ and $\be$ circles on $\Si_1*\Si_2$ are induced from the
corresponding objects in $\Si_1$ and $\Si_2$. The Heegaard diagram
$\mc{H}_1*\mc{H}_2$, in the case when $n=2$, looks like Figure
\ref{fig:combinedheegaard}, and in the case $n=3$, looks like the
third figure in the top row of
Figure~\ref{fig:modified-diagrams}. Since the modified diagrams
$\mc{H}''_1$ and $\mc{H}''_2$ are admissible, it follows
that the corresponding modified diagram $(\mc{H}_1*\mc{H}_2)''$ (third
figure in the third row of Figure~\ref{fig:modified-diagrams}), and
hence $\mc{H}_1*\mc{H}_2$ itself, is also admissible.

\section{Certain local isotopies}\label{sec:3}

Let $\mc{H}_{\al\be}=(\Si_{(g)},\al^{(g+k-1)},\be^{(g+k-1)},w^{(k)})$ be a Heegaard diagram for $S^3$, which possibly is non-admissible, and $S\sbs\Si$ be an open subsurface. Let $\mc{A}_{\mc{H}_{\al\be},S}\subseteq\mc{G}_{\mc{H}_{\al\be}}$ be the set of all the generators, none of whose coordinates lie inside $S$, and let $\mc{B}_{\mc{H}_{\al\be},S}=\mc{G}_{\mc{H}_{\al\be}}\sm \mc{A}_{\mc{H}_{\al\be},S}$ denote the rest of the generators. Let us assume that $S$ contains a disk $D$ that looks like the first part of Figure \ref{fig:isotopy} (with the train track convention): there are $b$ $\be$ arcs, all parallel to each other, with $b\geq 1$; there are $a_1+1+a_2$ $\al$ arcs, all parallel to each other, with $a_1,a_2\geq 0$, such that $a_2$ of them are disjoint from the $\be$ arcs, and each of the $a_1+1$ others, intersect each of the $b$ $\be$ arcs in exactly two points; there is a $w$ marking, such that the oriented boundary of the component of $D\sm(\al\cup\be)$ containing the $w$ marking, is an $\al$ arc followed by a $\be$ arc followed by an arc in $\del D$. Note that these $b$ $\be$ arcs need not belong to $b$ different $\be$ circles, and these $a_1+1+a_2$ $\al$ arcs need not belong to $a_1+1+a_2$ different $\al$ circles.

\begin{figure}
\centering
\includegraphics{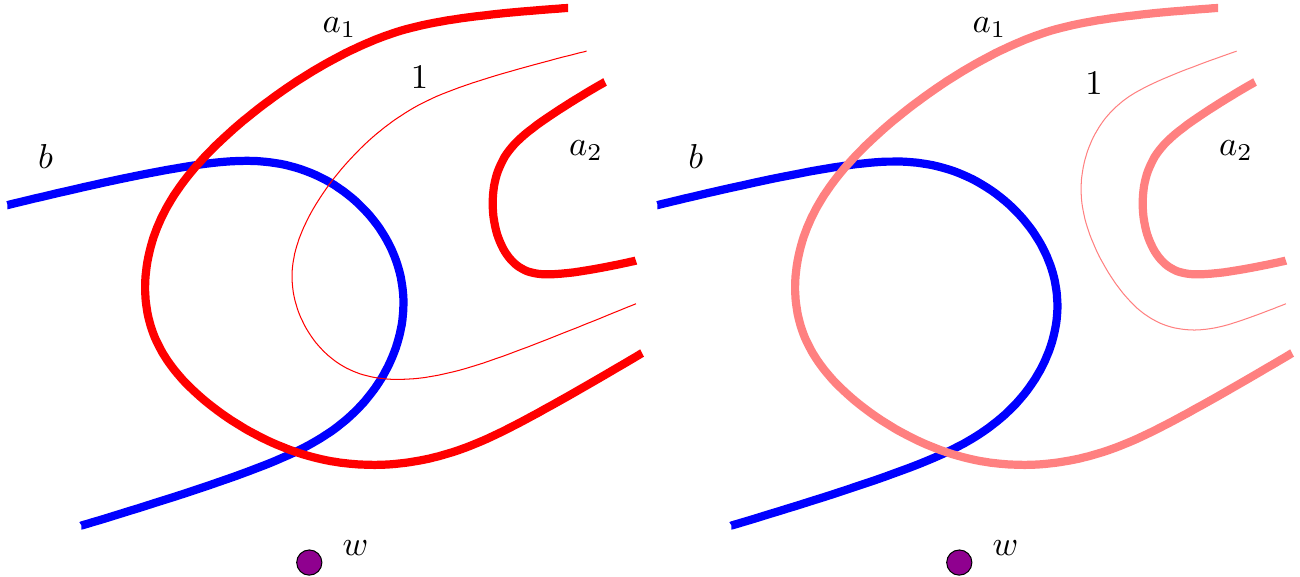}
\caption{\textbf{The disk $D$ in the Heegaard diagrams
    $\mc{H}_{\al\be}$ and $\mc{H}_{\ga\be}$.} The $\alpha$, $\beta$,
  and $\gamma$ arcs are represented by red, blue, and pink train
  tracks, respectively, with thin lines denoting curves. The number of
  arcs in each train track is also shown.}\label{fig:isotopy}
\end{figure}

Let $\mc{H}_{\ga\be}=(\Si,\ga,\be,w)$ be the Heegaard diagram (also possibly non-admissible) obtained from $\mc{H}$ after the local isotopy as shown in Figure \ref{fig:isotopy}. The surface $\Si$, the $\be$ circles, the $w$ markings, and the subsurface $S$ are unchanged. The $\al$ circles are replaced by the $\ga$ circles, which for the most part, are small perturbations of the corresponding $\al$ circles, except for one of the arcs in the disk $D$. There are $a_1+1+a_2$ $\ga$ arcs in $D\sbs S\sbs \Si$ in $\mc{H}_{\ga\be}$, of which $1+a_2$ of them are disjoint from the $\be$ arcs, and each of the remaining $a_1$ of them intersect each of the $b$ $\be$ arcs in exactly two points. This is shown in the second part of Figure \ref{fig:isotopy}, with the $\ga$ train tracks and arcs being denoted by thick and thin pink lines respectively. Once again, let $\mc{A}_{\mc{H}_{\ga\be},S}\sbs\mc{G}_{\mc{H}_{\ga\be}}$ be the set of all the generators, none of whose coordinates lie inside $S$, and let $\mc{B}_{\mc{H}_{\ga\be},S}=\mc{G}_{\mc{H}_{\ga\be}}\sm \mc{A}_{\mc{H}_{\ga\be},S}$ denote the rest of the generators. There is an obvious bijection $\mc{A}_{\mc{H}_{\al\be},S}\stackrel{\cong}{\longrightarrow} \mc{A}_{\mc{H}_{\ga\be},S}$, and we will always implicitly identify them by this bijection; there is an obvious injection $\mc{B}_{\mc{H}_{\ga\be},S}\hookrightarrow \mc{B}_{\mc{H}_{\al\be},S}$, and we will always implicitly treat $\mc{B}_{\mc{H}_{\ga\be},S}$ as a subset of $\mc{B}_{\mc{H}_{\al\be},S}$ by this injection.

\begin{prop}\label{prop:aba'b'}
If there are no empty positive domains from $\mc{A}_{\mc{H}_{\al\be},S}$ to
$\mc{B}_{\mc{H}_{\al\be},S}$, then there are no empty positive domains from
$\mc{A}_{\mc{H}_{\ga\be},S}$ to $\mc{B}_{\mc{H}_{\ga\be},S}$.
\end{prop}

\begin{proof}  
  We will prove the contrapositive of the statement. The basic idea is that any positive domain  from $\mc{A}_{\mc{H}_{\ga\be},S}$ to $\mc{B}_{\mc{H}_{\ga\be},S}$  induces a corresponding positive domain in the original diagram, simply by tracing multiplicities through the reversal of the isotopy.  To make this precise, let us assume
  that $D_1\in\mc{D}_0(x,y)$ is a positive domain, from some generator
  $x\in \mc{A}_{\mc{H}_{\ga\be},S}$ to some generator
  $y\in \mc{B}_{\mc{H}_{\ga\be},S}$. Let
  $\bar{x}\in \mc{A}_{\mc{H}_{\al\be},S}$ and
  $\bar{y}\in \mc{B}_{\mc{H}_{\al\be},S}$ be the images of $x$ and $y$
  under the bijection
  $\mc{A}_{\mc{H}_{\ga\be},S}\stackrel{\cong}{\longrightarrow}
  \mc{A}_{\mc{H}_{\al\be},S}$ and the injection
  $\mc{B}_{\mc{H}_{\ga\be},S}\hookrightarrow
  \mc{B}_{\mc{H}_{\al\be},S}$, respectively. The empty domain $D_1$
  gives rise to another empty domain
  $D_2\in\mc{D}_0(\bar{x},\bar{y})$. The local coefficients of $D_1$
  and $D_2$ are shown in Figure~\ref{fig:coefficients}. We will simply
  show that $D_2$ is also a positive domain.

\begin{figure}
\centering
\includegraphics{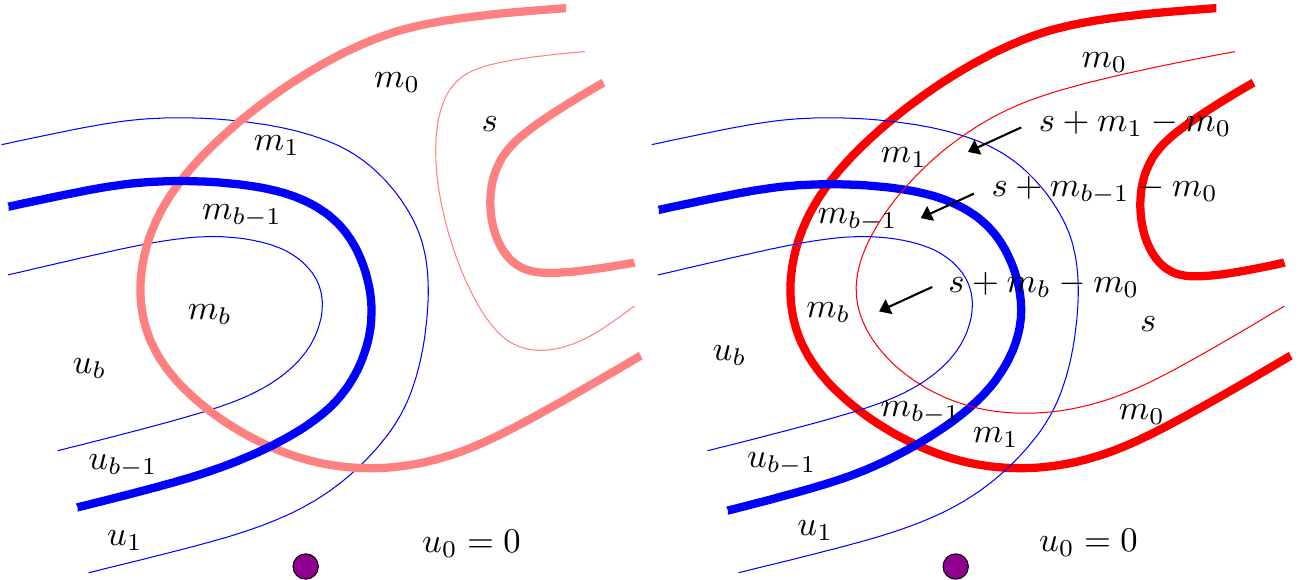}
\caption{\textbf{The local coefficients of $D_1$ and $D_2$ in
    $\mc{H}_{\ga\be}$ and $\mc{H}_{\al\be}$.} We prove that if $D_1$
  is a positive domain, then so is $D_2$, which follows once we show that $s+m_i-m_0\geq s+u_i$.}\label{fig:coefficients}
\end{figure}

Towards this end, we will prove that none of the coefficients
$m_1,\ldots, m_b$ are smaller than the coefficient $m_0$. We will
prove this by showing that $u_i\leq m_i-m_0$, for all $0\leq i\leq
b$. Since each $u_i\geq 0$, this will complete the proof.

We will prove $u_i\leq m_i-m_0$ by an induction on $i$. Since,
$u_0=0=m_0-m_0$, the base case is trivial. Assume by induction that
the statement is true for $i$. Before we prove the statement for
$i+1$, let us make one small observation about the domain $D_1$.

Since $\del(\del D_1\cap{\be})=x-y$, and since none of the coordinates
of $x$ lie in the disk $D$, we therefore have that $\del(\del D_1\cap{\be})$, viewed
as a $0$-chain, does not contain any point with positive sign in the
local neighborhood $D$. Let $\tau$ be an oriented arc, which is a
subspace of a $\be$ circle, and is supported entirely inside $D$. Let
the coefficient of $\del D_1\cap{\be}$ near the beginning of $\tau$ be
$c>0$. The observation that $\del(\del D_1\cap{\be})$ does not contain
any positively signed point in $D$, implies that the coefficient of
$\del D_1\cap{\be}$ near the end of $\tau$ is greater than or equal to
$c$. We summarize this observation by the statement that ``$\del
D_1\cap{\be}$ does not stop inside $D$''.

\begin{figure}
\centering
\includegraphics{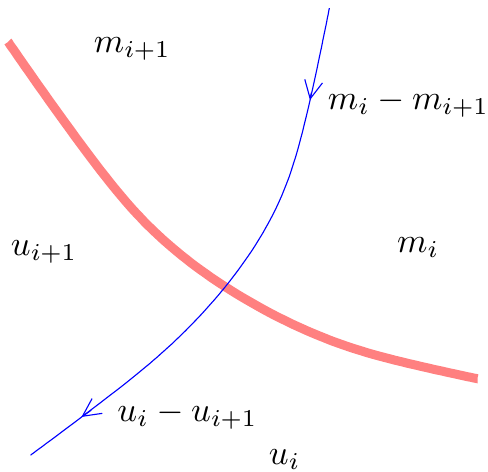}
\caption{\textbf{The induction step.} Assuming $u_i\leq m_i-m_0$, we
  prove $u_{i+1}\leq m_{i+1}-m_0$ by showing
  $u_{i+1}-u_i\leq m_{i+1}-m_i$. The coefficients of $D_1$ and the
  coefficients of $\del D_1\cap{\beta}$ are shown.}\label{fig:induct}
\end{figure}

Now, we are all set to prove the induction statement for $i+1$. In the
Heegaard diagram $\mc{H}_{\ga\be}$, let $\be_1$ be the $\be$ circle that
separates the elementary region with coefficient $m_{i+1}$ from the
elementary region with coefficient $m_i$, and let $\tau$ be an
oriented subarc of $\be_1$, running from the elementary region with
coefficient $m_i$ to the elementary region with coefficient
$u_i$. Therefore, the signed coefficients of $\del D_1\cap{\be}$ near
the beginning and the end of $\tau$ are $m_i-m_{i+1}$ and
$u_i-u_{i+1}$ respectively, as shown in Figure \ref{fig:induct}. In light of the observation above that $\del D_1\cap{\be}$ does not stop inside $D$, it follows that coefficient  of $\del D_1\cap{\be}$ near the end of $\tau$ is greater than or equal to that at the end:
\[ m_i- m_{i+1}\le u_i-u_{i+1}.\]
Subtracting $m_0$ from both sides of the inequality, and rearranging, we have:
\[ u_{i+1} + m_i  - m_0 - u_i \le m_{i+1}-m_0  \]
But the inductive hypothesis is that $0\le m_i  - m_0 - u_i$, hence $u_{i+1}\le m_{i+1}-m_0$, as desired.
\end{proof}

\begin{figure}
\centering
\includegraphics{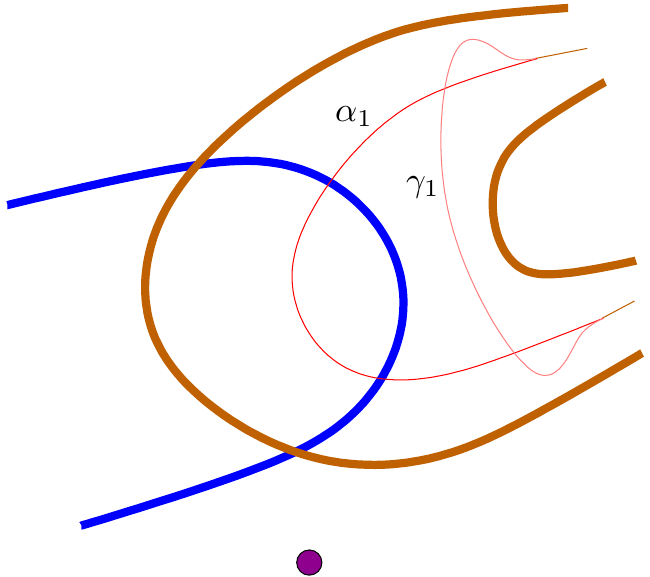}
\caption{\textbf{The Heegaard diagram $\mc{H}_{\al\be\ga}$ in the
    neighborhood $D$.} As before, $\alpha$, $\beta$, $\gamma$ are red,
  blue, and pink, respectively. A thin brown curve denotes a train
  track which is pair of parallel $\al$ and $\ga$ curves, and a thick
  brown curve is a train track of such train
  tracks.}\label{fig:triangle}
\end{figure}

Now, we prove a similar statement for triangular domains. Let
$\mc{H}_{\al\be\ga}=(\Si,\al,\be,\ga,w)$ be the triple Heegaard
diagram, obtained by combining the above two diagrams; it is also
possibly non-admissible. We assume that $\ga_i$ is a small translate of $\al_i$, intersecting it transversely in exactly two points, so that $\ga_i$ is disjoint from
$\al_j$ for $i\neq j$. Let
$\al_1$ be the $\al$ circle that is changed to the $\ga$ circle
$\ga_1$ in Figure \ref{fig:triangle}. Therefore, a neighborhood $N_i$
of $\al_i\cup\ga_i$ for $i\neq 1$ looks like the first part of Figure
\ref{fig:alge}, with none of the two intersection points in
$\al_i\cap\ga_i$ lying in the neighborhood $D$. A neighborhood
$N_1$ of $\al_1\cup\ga_1$ looks like the second part of Figure
\ref{fig:alge}, with both the intersection points in $\al_1\cap\ga_1$
lying in the neighborhood $D$. The coordinates of the top generator
$\theta$ are shown.

\begin{figure}
\centering
\includegraphics{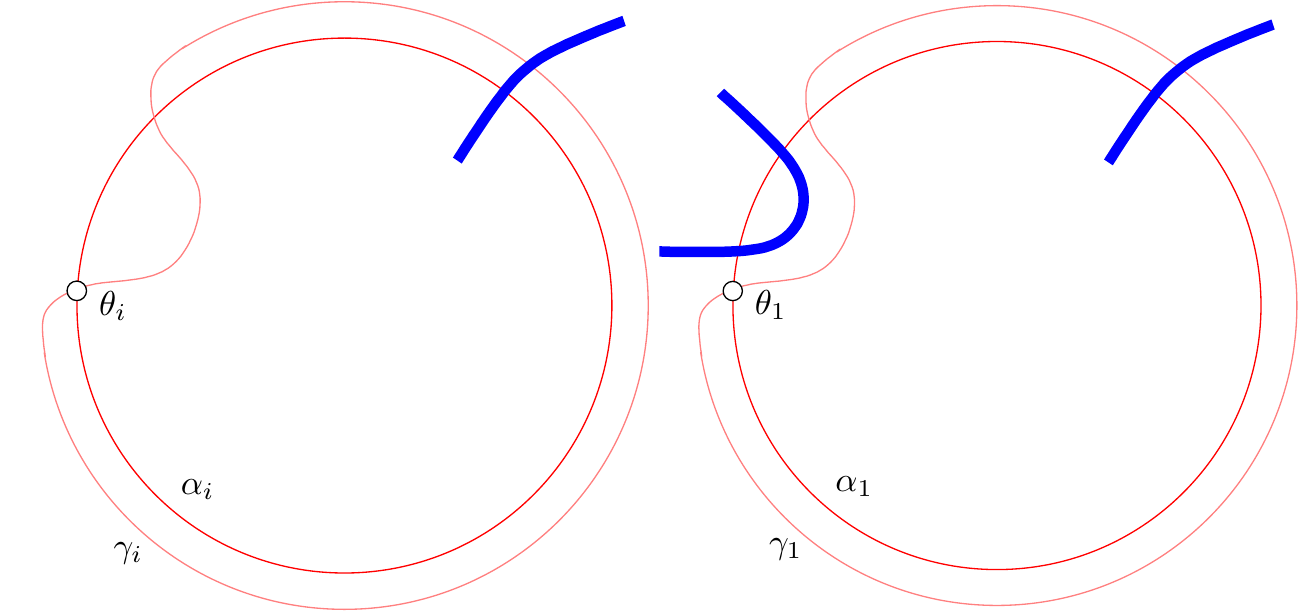}
\caption{\textbf{Neighborhoods $N_i$ of $\al_i\cup\ga_i$.} For
  $i\neq 1$ (left), the small bigon region is disjoint from $D$, while
  for $i=1$ (right), the small bigon region is contained inside
  $D$. The coordinates of the top generator $\theta$ are shown by
  white dots.}\label{fig:alge}
\end{figure}

\begin{prop}\label{prop:abab'}
If there are no empty positive domains from $\mc{A}_{\mc{H}_{\al\be},S}$ to
$\mc{B}_{\mc{H}_{\al\be},S}$, then there are no empty positive triangular domains from
$\mc{A}_{\mc{H}_{\al\be},S}$ to $\mc{B}_{\mc{H}_{\ga\be},S}$.
\end{prop}

\begin{proof}
  Let $D_1\in\mc{T}_0(x,y)$ be a positive triangular domain, for some
  $x\in \mc{A}_{\mc{H}_{\al\be},S}$ and
  $y\in \mc{B}_{\mc{H}_{\ga\be},S}$. Let
  $\bar{y}\in \mc{B}_{\mc{H}_{\al\be},S}$ be the image of $y$ under
  the injection
  $\mc{B}_{\mc{H}_{\ga\be},S}\hookrightarrow
  \mc{B}_{\mc{H}_{\al\be},S}$. It is easy to see that there is a
  unique empty triangular domain $D_2\in\mc{T}_0(\bar{y},y)$, whose
  non-zero coefficients are supported inside the neighborhoods $N_i$,
  such that $\del D_2\cap{\ga}=\del D_1\cap{\ga}$. Then, the $2$-chain
  $D_3=D_1-D_2$ is a domain in $\mc{D}_0(x,\bar{y})$. We will 
  show that $D_3$ is also a positive domain,  establishing  the
  contrapositive of the given statement.

  The coefficients of $D_2$ are zero outside $\cup_i N_i$, and the
  neighborhoods $N_i$ for $i\neq 1$ can be considered as special
  cases of the neighborhood $N_1$. Therefore, we only need to
  concentrate on the coefficients of $D_3$ in $N_1$. Figure
  \ref{fig:alger} shows the coefficients of $D_1$, $\del D_1\cap{\ga_1}$
  and $D_3$ in this neighborhood. The coordinates of $\theta$ and $y$
  on $\ga_1$ are shown. The coefficient of $\del D_1\cap{\gamma_1}$ is
  either $r$ or $r+1$ (which need not be positive), as shown.

\begin{figure}
\centering
\includegraphics{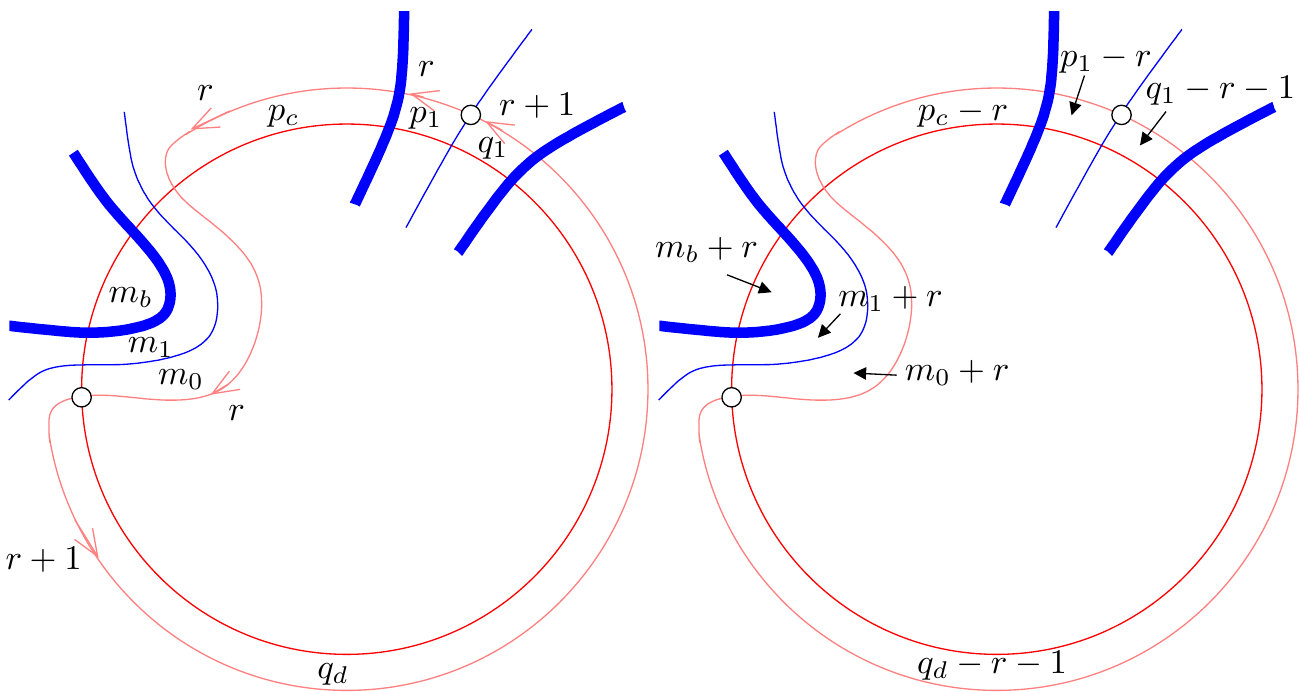}
\caption{\textbf{Coefficients of $D_1$ (left) and $D_3$ (right) in the
    neighborhood $N_1$ of $\al_1\cup\ga_1$.} The coordinates of
  $\theta $ and $y$ are shown by white dots. The coefficients of
  $\del D_1\cap{\ga_1}$ are also shown on the left.}\label{fig:alger}
\end{figure}

Since $D_1$ is a positive domain, and since the region with
coefficient $q_\ell$ in $D_1$ has $\gamma_1$ on its boundary with
coefficient $r+1$, we must have $q_\ell\geq r+1$ for all
$1\leq\ell\leq d$. Similarly, we must have $p_\ell\geq r$ for all
$1\leq\ell \leq c$ and $m_0\geq -r$. Therefore, in order to show that
$D_3$ is also a positive domain, we only need to show that each of the
coefficients $m_1,\ldots,m_b$ are greater than or equal to
$-r$. However, exactly as in the proof of
Proposition~\ref{prop:aba'b'}, using the fact that $x$ has no
coordinates inside the disk $D$, we can show that none of the
coefficients $m_1,\ldots,m_b$ are smaller than $m_0$, and this
completes the proof.
\end{proof}

Let us henceforth assume that the Heegaard diagram $\mc{H}_{\ga\be}$
is admissible. Then $\mc{H}_{\al\be}$ and $\mc{H}_{\al\be\ga}$ are
also admissible, and in that case, $\tCF_{\mc{H}_{\al\be}}$ and
$\tCF_{\mc{H}_{\ga\be}}$ are the chain complexes, freely generated
over $\F_2$, by $\mc{G}_{\mc{H}_{\al\be}}$ and
$\mc{G}_{\mc{H}_{\ga\be}}$, respectively. Let
$\tSF_{\mc{H}_{\al\be},S}$ and $\tSF_{\mc{H}_{\ga\be},S}$ be the
$\F_2$-submodules, freely generated by $\mc{A}_{\mc{H}_{\al\be},S}$
and $\mc{A}_{\mc{H}_{\ga\be},S}$, respectively.

\begin{thm}\label{thm:simplify}
Assume that there are no empty positive domains from $\mc{A}_{\mc{H}_{\al\be},S}$
to $\mc{B}_{\mc{H}_{\al\be},S}$. Then $\tSF_{\mc{H}_{\al\be},S}$ is a subcomplex of
$\tCF_{\mc{H}_{\al\be}}$, $\tSF_{\mc{H}_{\ga\be},S}$ is a subcomplex of
$\tCF_{\mc{H}_{\ga\be}}$, and the chain map from $\tCF_{\mc{H}_{\al\be}}$ to
$\tCF_{\mc{H}_{\ga\be}}$ induces a chain map from $\tSF_{\mc{H}_{\al\be},S}$ to
$\tSF_{\mc{H}_{\ga\be},S}$. Furthermore, the chain maps
$\tCF_{\mc{H}_{\al\be}}\rightarrow\tCF_{\mc{H}_{\ga\be}}$ and
$\tSF_{\mc{H}_{\al\be},S}\rightarrow\tSF_{\mc{H}_{\ga\be},S}$ are quasi-isomorphisms.
\end{thm}

\begin{proof}
Proposition~\ref{prop:aba'b'} implies that there are no empty positive
domains from $\mc{A}_{\mc{H}_{\ga\be},S}$ to $\mc{B}_{\mc{H}_{\ga\be},S}$, and Proposition~\ref{prop:abab'} implies that there are no empty positive triangular
domains from $\mc{A}_{\mc{H}_{\al\be},S}$ to $\mc{B}_{\mc{H}_{\ga\be},S}$. Since the
non-zero terms in the boundary maps on $\tCF_{\mc{H}_{\al\be}}$ and
$\tCF_{\mc{H}_{\ga\be}}$ come only from empty positive domains, and the
non-zero terms in the chain map from $\tCF_{\mc{H}_{\al\be}}$ to
$\tCF_{\mc{H}_{\ga\be}}$ come only from empty positive triangular domains,
$\tSF_{\mc{H}_{\al\be},S}\hookrightarrow\tCF_{\mc{H}_{\al\be}}$ is a subcomplex,
$\tSF_{\mc{H}_{\ga\be},S}\hookrightarrow\tCF_{\mc{H}_{\ga\be}}$ is a subcomplex,
and the chain map $\tCF_{\mc{H}_{\al\be}}\rightarrow\tCF_{\mc{H}_{\ga\be}}$
induces a chain map $\tSF_{\mc{H}_{\al\be},S}\rightarrow\tSF_{\mc{H}_{\ga\be},S}$,
resulting in the following commuting square. 
\[
\begin{tikzpicture}[xscale=2]
\node (sf-ab) at (0,0) {$\tSF_{\mc{H}_{\al\be},S}$};
\node (sf-cb) at (0,-1) {$\tSF_{\mc{H}_{\ga\be},S}$};
\node (cf-ab) at (1,0) {$\tCF_{\mc{H}_{\al\be}}$};
\node (cf-cb) at (1,-1) {$\tCF_{\mc{H}_{\ga\be}}$};

\draw[right hook->](sf-ab) -- (cf-ab);
\draw[right hook->](sf-cb) -- (cf-cb);
\draw[->](sf-ab) -- (sf-cb);
\draw[->](cf-ab) -- (cf-cb);
\end{tikzpicture}
\]

We will now show that the vertical arrows induce isomorphisms on
homology. A \emph{$2$-cochain} on a Heegaard diagram
or a triple Heegaard diagram is a map which assigns real numbers to
the elementary regions; a \emph{non-negative $2$-cochain} is a
$2$-cochain which only assigns non-negative numbers; and a
\emph{positive $2$-cochain} is a $2$-cochain which only assigns
positive numbers. Since $\mc{H}_{\ga\be}$ is admissible, by
\cite[Lemma 4.12]{POZSz}, there exists a positive $2$-cochain $C_0$ on
$\mc{H}_{\ga\be}$, which evaluates to zero on all empty periodic
domains in $\mc{H}_{\ga\be}$.

A cochain in $\mc{H}_{\al\be\ga}$ induces cochains in
$\mc{H}_{\al\be}$ and $\mc{H}_{\ga\be}$, by forgetting the $\ga$
circles and the $\al$ circles, respectively, as well as the
coefficients of the cochain on the thin elementaty regions that lie
entirely inside the neighbordhoods $N_i$. We will now construct a
non-negative cochain $C$ on $\mc{H}_{\al\be\ga}$, such that: $C$
assigns zero precisely to the elementary regions that lie entirely in
the neighborhoods $N_i$; and $C$ induces the positive cochain $C_0$ in
$\mc{H}_{\ga\be}$. Since the empty periodic domains in
$\mc{H}_{\al\be\ga}$ are generated by the empty periodic domains in
$\mc{H}_{\ga\be}$ and the periodic domains that are supported in
$\cup_i N_i$, this would imply that $C$ evaluates to zero on any empty
periodic domain in $\mc{H}_{\al\be\ga}$. The way to construct $C$ is
fairly straightforward. Let $R$ be an elementary region in
$\mc{H}_{\ga\be}$, and let $r$ be the assignment of $C_0$ on $R$. The
region $R$ might get cut up into several elementary regions
$R_1,\ldots,R_n$ in $\mc{H}_{\al\be\ga}$, and some of them might lie
entirely in the neighborhoods $N_i$, but at least one of them does
not. Choose non-negative real numbers $r_1,\ldots,r_n$, such that,
$\sum_i r_i=r$ and $r_i=0$ if and only if $R_i$ lies entirely in
$\cup_i N_i$. Then assign the number $r_i$ to the elementary region
$R_i$ in the $2$-cochain $C$.

This non-negative $2$-cochain $C$ gives rise to filtrations on the
mapping cones
$\tSF_{\mc{H}_{\al\be},S}\rightarrow\tSF_{\mc{H}_{\ga\be},S}$
and $\tCF_{\mc{H}_{\al\be}}\rightarrow\tCF_{\mc{H}_{\ga\be}}$,
as follows: given any two generators
$x,y\in\mc{G}_{\mc{H}_{\al\be}}\cup\mc{G}_{\mc{H}_{\ga\be}}$, the
relative filtration grading between them is $\langle C,D\rangle$, for
any $D$ in $\mc{D}_0(x,y)$ or $\mc{T}_0(x,y)$ as the case may be. On
the associated graded level, we only count domains or triangular
domains that lie entirely inside these neighborhoods $N_i$, and then
it is a fairly straightforward to check that the associated graded
maps on the associated graded objects are isomorphisms, and therefore,
the original chain maps must have been quasi-isomorphisms as well.
\end{proof}

\section{Main theorems}\label{sec:4}

\begin{proof}[Proof of Theorem \ref{thm1}]
  Following the notations from Section~\ref{subsec:murasugi}, let
  $\mc{H}_1$, $\mc{H}_2$ and $\mc{H}_1*\mc{H}_2$, as shown in Figures
  \ref{fig:bothlinksdiagram} and \ref{fig:combinedheegaard}, be the
  Heegaard diagrams adapted to Seifert surfaces $R_1$, $R_2$ and
  $R_1*R_2$ for the links $L_1$, $L_2$ and $L=L_1*L_2$,
  respectively. The corresponding Heegaard surfaces contain embedded
  subsurfaces $S_1$, $S_2$ and $S_1*S_2$, which represent $R_1$, $R_2$
  and $R_1*R_2$, respectively. Furthermore, $\mc{H}_1$ has
  $(l_1+\delta_1)$ $w$-markings, $\mc{H}_2$ has $(l_2+\delta_2)$
  $w$-markings, $\mc{H}_1*\mc{H}_2$ has $(l+\delta)$ $w$-markings.

\begin{figure}
\includegraphics{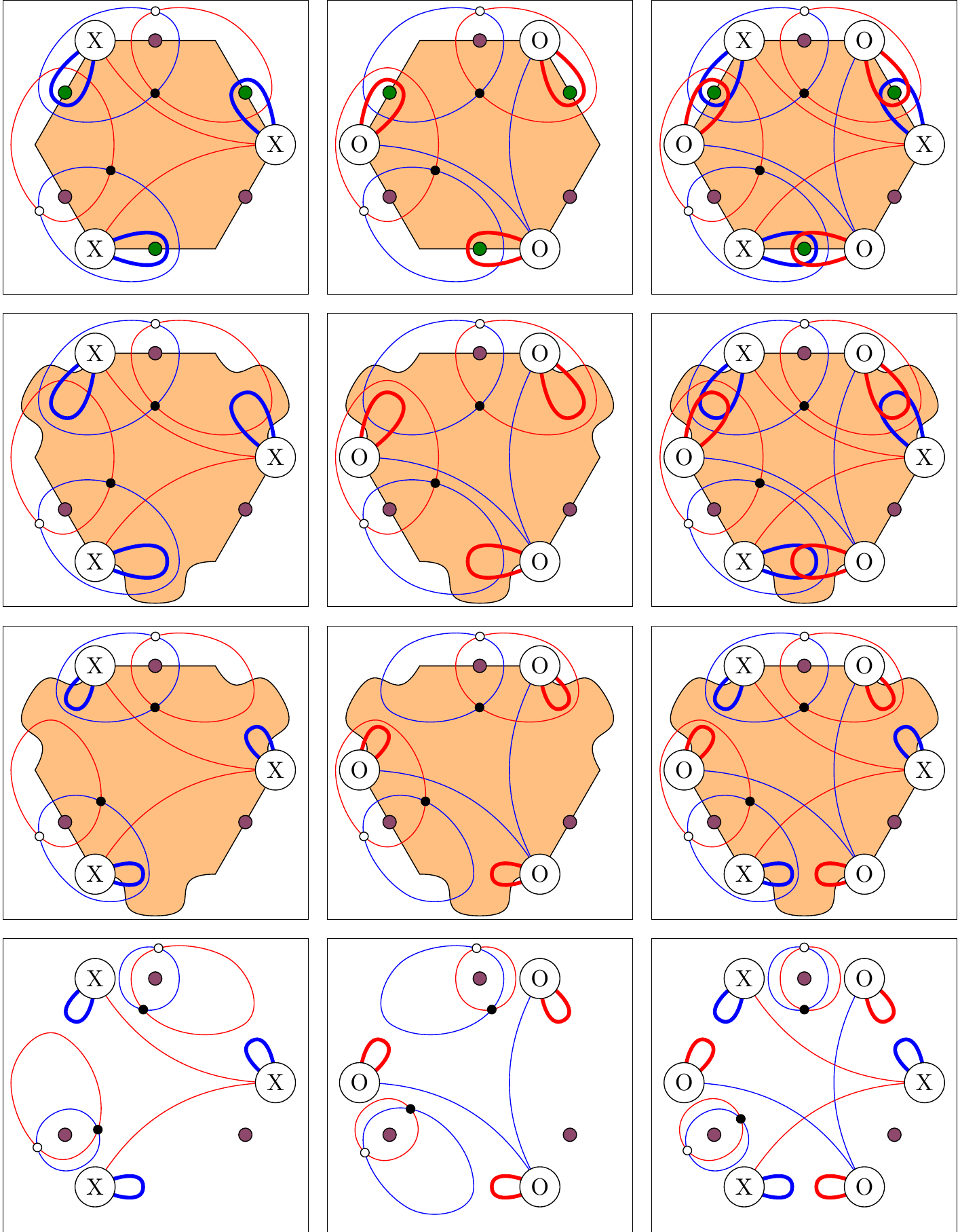}
\caption{\textbf{The Heegaard diagrams appearing in the proof of
    Theorems~\ref{thm1} and~\ref{thm2}.} The left, middle, right
  columns represent diagrams obtained from $\mc{H}_1$, $\mc{H}_2$,
  $\mc{H}_1*\mc{H}_2$, respectively; the consecutive rows represent
  the diagrams $\mc{H}$, $\mc{H}'$, $\mc{H}''$, and $\mc{H}'''$; X and
  O denote handles going down (into the page) and up (towards the
  reader), respectively.}\label{fig:modified-diagrams}
\end{figure}
Thanks to Corollary~\ref{cor:2-step-filtered}, we only need to produce an isomorphism of graded chain complexes
\begin{align*}
\mc{F}_{\mc{H}_1*\mc{H}_2}&\big(-\frac{1}{2}(l+2\delta-\chi(R_1*R_2))\big)[l+\delta-1]\\&\cong \mc{F}_{\mc{H}_1}\big(-\frac{1}{2}(l_1+2\delta_1-\chi(R_1))\big)[l_1+\delta_1-1]\otimes
\mc{F}_{\mc{H}_2}\big(-\frac{1}{2}(l_2+2\delta_2-\chi(R_2))\big)[l_2+\delta_2-1],
\end{align*}
or in terms of the notation from Section~\ref{sec:3}, since $(l+\delta-n)=(l_1+\delta_1-n)+(l_2+\delta_2-n)$,
\begin{equation}\label{eq:graded-isom}
\tSF_{\mc{H}_1*\mc{H}_2,S_1*S_2}[n-1]\cong\tSF_{\mc{H}_1,S_1}[n-1]\otimes\tSF_{\mc{H}_2,S_2}[n-1].
\end{equation}

Let $(\mc{H},{S})$ denote any of $(\mc{H}_1,S_1)$, $(\mc{H}_2,{S}_2)$
or $(\mc{H}_1*\mc{H}_2,S_1*S_2)$, as shown in the first row of
Figure~\ref{fig:modified-diagrams}. The sphere $S^2$ of the Murasugi
sum is represented by the squares on the page, with $L_1$ lying below
the page and $L_2$ lying above. Let $\Si$ be the Heegaard surface and
let $\Si^1$ (respectively, $\Si^2$) be the portion of $\Si$ that lies
inside (respectively, outside) the sphere $S^2$.

Let $\mc{H}'$ be the Heegaard diagrams for $S^3$ obtained from
$\mc{H}$ by forgetting the $z$ markings, and let $S'$ be the open
subsurface of the Heegaard surface of $\mc{H}$, obtained by slightly
modifying $S$ in a neighborhood of the $2n$-gon $A_1A_2\dots A_{2n}$
so as to include all the intersections between $\alpha$ and $\beta$
circles near the erstwhile $z$-markings. These modified diagrams
$(\mc{H}',{S}')$ are shown in the second row of
Figure~\ref{fig:modified-diagrams}.

Since we have not changed the underlying Heegaard diagrams, clearly,
$\tCF_{\mc{H'}}=\tCF_{\mc{H}}$.  Next we claim that any generator
$x\in\mc{G}_{\mc{H}}$ that does not have any coordinate in $S$ can not
have any coordinate in $S'$ either. This is a simple counting
argument.  Indeed, let us say $\mc{H}_i$ has a total of $n_i$ $\alpha$-circles
and $n_i$ $\beta$-circles; of them, exactly $(n-1)$ $\alpha$-circles
and $(n-1)$ $\beta$-circles lie entirely inside the $S^2$. Then
$\mc{H}_1*\mc{H}_2$ has a total of $(n_1+n_2-n+1)$ $\alpha$-circles
and $(n_1+n_2-n+1)$ $\beta$-circles, again with exactly $(n-1)$
$\alpha$-circles and $(n-1)$ $\beta$-circles lying entirely inside the
$S^2$. For $\mc{H}\neq \mc{H}_2$, we see that $(n_1-n+1)$ $\alpha$
circles of $\mc{H}$ lie entirely within $\Si^1\cup S$, and so $x$ must
have at least $(n_1-n+1)$ coordinates in $\Si^1$ (as it avoids the surface $S$, by assumption). Similarly, for
$\mc{H}\neq \mc{H}_1$, we see that $(n_2-n+1)$ $\beta$ circles of
$\mc{H}$ lie entirely within $\Si^2\cup S$, and so $x$ must have at
least $(n_2-n+1)$ coordinates in $\Si^2$. Therefore, in all cases, $x$
has at most $(n-1)$ coordinates in the sphere $S^2$.  It follows that, in fact, $x$ must have exactly $(n-1)$ coordinates in the sphere, occupied by the $(n-1)$ $\alpha$ and $\beta$ circles that
lie entirely therein. Specifically, they must be the white dots as
shown in the first row of
Figure~\ref{fig:modified-diagrams}. Therefore, we get that
$\tSF_{\mc{H}',S'}=\tSF_{\mc{H},S}$. Moreover, since there are no
positive domains from generators of $\tSF_{\mc{H},S}$ to the other
generators of $\tCF_{\mc{H}}$ due to Alexander grading, there are no
positive domains from the generators of $\tSF_{\mc{H}',S'}$ to the
generators of $\tCF_{\mc{H'}}$ as well, and we have the
following identification of subcomplexes.
\begin{equation}\label{eq:trivial-isomorphism}
\vcenter{\hbox{\begin{tikzpicture}[xscale=2]
\node (sf-ab) at (0,0) {$\tSF_{\mc{H},S}$};
\node (sf-cb) at (0,-1) {$\tSF_{\mc{H}',S'}$};
\node (cf-ab) at (1,0) {$\tCF_{\mc{H}}$};
\node (cf-cb) at (1,-1) {$\tCF_{\mc{H}'}$};

\draw[right hook->](sf-ab) -- (cf-ab);
\draw[right hook->](sf-cb) -- (cf-cb);
\path (sf-ab) -- (sf-cb) node[midway,rotate=90] {=};
\path (cf-ab) -- (cf-cb) node[midway,rotate=90] {=};
\end{tikzpicture}}}
\end{equation}

As in Section~\ref{sec:3}, we perform local isotopies to separate the
$\alpha$ and $\beta$ curves in the neighborhood of the erstwhile
$z$-markings so that we obtain the Heegaard diagrams in the third row
of Figure~\ref{fig:modified-diagrams}, which we denote by $(\mc{H}'',S')$. The aforementioned isotopies are supported inside $S'$, and
there are no domains from the generators of $\tSF_{\mc{H}',S'}$ to the
other generators of $\tCF_{\mc{H}'}$.   Recalling that the Heegaard
diagrams were constructed to ensure that $\mc{H}''$ is admissible (see ~(\ref{item:make-stuff-admissible})), we see that the hypotheses of
Propositions~\ref{prop:aba'b'} and~\ref{prop:abab'} are satisfied.  Applying the propositions, we obtain
a quasi-isomorphism between the following two-step filtered complexes.
\begin{equation}\label{eq:technical-isomorphism}
\vcenter{\hbox{\begin{tikzpicture}[xscale=2]
\node (sf-ab) at (0,0) {$\tSF_{\mc{H}',S'}$};
\node (sf-cb) at (0,-1.5) {$\tSF_{\mc{H}'',S'}$};
\node (cf-ab) at (1,0) {$\tCF_{\mc{H}'}$};
\node (cf-cb) at (1,-1.5) {$\tCF_{\mc{H}''}$};

\draw[right hook->](sf-ab) -- (cf-ab);
\draw[right hook->](sf-cb) -- (cf-cb);
\draw[->] (sf-ab) -- (sf-cb) node[midway,anchor=east] {q.i.};
\draw[->] (cf-ab) -- (cf-cb) node[midway,anchor=west] {q.i.};
\end{tikzpicture}}}
\end{equation}

Next we claim that for any generator $x$ of $\tCF_{\mc{H}''}$, all its
coordinates in $S^2$ must be the white or black dots from the third
row of Figure~\ref{fig:modified-diagrams}. It is once again a counting
argument, but with the roles of $\al$ and $\be$ reversed. For
$\mc{H}''\neq\mc{H}''_2$, we see that $(n_1-n+1)$ $\beta$ circles of
$\mc{H}''$ have no intersections with $\alpha$ circles in $S^2$, so
$x$ must have at least $(n_1-n+1)$ coordinates in $\Si^1$. Similarly,
for $\mc{H}''\neq\mc{H}''_1$, we see that $(n_2-n+1)$ $\alpha$ circles
of $\mc{H}''$ have no intersections with $\beta$ circles in $S^2$, so
$x$ must have at least $(n_2-n+1)$ coordinates in $\Si^2$. In all
cases, $x$ has at most $(n-1)$ coordinates in the sphere $S^2$, and
therefore, they must be occupied by the $(n-1)$ $\alpha$ and $\beta$
circles that lie entirely in $S^2$; moreover, they must be the white
or black dots as shown in the third row of
Figure~\ref{fig:modified-diagrams}.

After numbering the $\alpha$ circles that lie entirely in $S^2$
arbitrarily (but consistently across Heegaard diagrams) from $1$ to
$n-1$, each generator $x$ can be represented as a pair
$(\vec{a},x^o)$, where $\vec{a}=(a_1,\dots,a_{n-1})\in\{0,1\}^{n-1}$
with $a_i=0$ if and only if $\alpha_i$ contains the white dot, and
$x^o$ denotes the  coordinates of $x$ that lie outside
$S^2$. Consider the usual partial order on $\{0,1\}^{n-1}$ with
$\vec{a}\leq\vec{b}$ if $a_i\leq b_i$ for all $i$, and the usual
$L^1$-norm on $\{0,1\}^{n-1}$ given by $|\vec{a}|=\sum_ia_i$. Since empty
positive domains are not allowed to pass through the $w$ markings, we
see that if $D\in\mc{D}_0((\vec{a},x^o),(\vec{b},y^o))$ is an empty
positive domain, then $\vec{b}\leq \vec{a}$. Let $\tCF^W_{\mc{H}''}$
denote the subcomplex of $\tCF_{\mc{H}''}$ spanned by generators with
$\vec{a}=0$, that is, generators with only white dots. Then we have
nested subcomplexes
\[
\tSF_{\mc{H}'',S'}\hookrightarrow\tCF^W_{\mc{H}''}\hookrightarrow\tCF_{\mc{H}''}.
\]
For any generator $x$ of $\tCF_{(\mc{H}_1*\mc{H}_2)''}$, let $x^i$ be
all its coordinates that lie in $\Si^i$. Then the map
\[
(\vec{0},x^1,x^2)\mapsto (\vec{0},x^1)\otimes(\vec{0},x^2)
\]
produces the following identification between the following chain groups:
\begin{equation}\label{eq:main-isomorphisms}
\vcenter{\hbox{\begin{tikzpicture}[xscale=4]
\node (sf-ab) at (0,0) {$\tSF_{(\mc{H}_1*\mc{H}_2)'',(S_1*S_2)'}$};
\node (sf-cb) at (0,-1.5) {$\tSF_{\mc{H}''_1,S'_1}\otimes\tSF_{\mc{H}''_2,S'_2}$};
\node (cf-ab) at (1,0) {$\tCF^W_{(\mc{H}_1*\mc{H}_2)''}$};
\node (cf-cb) at (1,-1.5) {$\tCF^W_{\mc{H}''_1}\otimes\tCF^W_{\mc{H}''_2}$};

\draw[right hook->](sf-ab) -- (cf-ab);
\draw[right hook->](sf-cb) -- (cf-cb);
\draw[->] (sf-ab) -- (sf-cb) node[midway,anchor=east] {$\cong$};
\draw[->] (cf-ab) -- (cf-cb) node[midway,anchor=west] {$\cong$};
\end{tikzpicture}}}
\end{equation}

Let us now prove that the vertical arrows are relative Maslov grading
preserving chain maps---that is, the above identifications are
identifications of chain complexes, up to a single absolute Maslov
grading shift. Let $(\vec{0},x^1,x^2)$ and $(\vec{0},y^1,y^2)$ be two
generators of $\tCF^W_{(\mc{H}_1*\mc{H}_2)''}$ and let
$D\in\mc{D}((\vec{0},x^1,x^2),(\vec{0},y^1,y^2))$ be an empty positive
domain in $(\mc{H}_1*\mc{H}_2)''$ connecting them. Such a domain has to be disjoint from $S^2$. 
Indeed, the fact  the domain has no corner points in $S^2$ forces it to restrict to a periodic domain therein, which the admissibility condition then ensures is the trivial domain with zero multiplicities. It follows that $D$ splits as a disjoint union $D^1\cup D^2$
of empty positive domains, with
$D^i\in\mc{D}((\vec{0},x^i),(\vec{0},y^i))$ in $\mc{H}_i''$. Recall from
Section~\ref{subsec:knotfloer} that $D$ can only contribute to the
differential if one of $D^1$ and $D^2$ is the trivial
domain. Furthermore, since such domains avoid $S^2$, we may choose
complex structures (and their perturbations) for the three Heegaard
diagrams so that they agree on $\Sigma^1$ and $\Sigma^2$ (and their
corresponding symmetric products); therefore, if $D^1$ (respectively,
$D^2$) is the trivial domain, then the contribution of $D$ will agree
with the contribution of $D^2$ (respectively, $D^1$). This is an instance of the {\em localization principle} \cite[Section 9.4]{JR}, and it establishes that
the vertical arrows are chain maps, and indeed chain isomorphisms.  To see that the vertical arrows
also respect the relative Maslov gradings, consider generators
$(0,x^i),(0,y^i)$ of $\tCF^W_{\mc{H}_i''}$, and choose empty domains
$D^i$ in $\mc{H}_i''$ (not necessarily positive) connecting them.  In
$\mc{H}_1''$ (respectively, $\mc{H}_2''$) consider the $(n-1)$
$\alpha$ (respectively, $\beta$) circles that lie entirely inside
$S^2$; each of them bounds a disk also entirely inside $S^2$, and each
such disk is comprised of two bigon-shaped elementary regions---one
containing a basepoint, and one without. By adding some number of
copies of these disks to the domain $D^i$, we can get a domain $E^i$
(not necessarily empty) connecting $(0,x^i)$ to $(0,y^i)$ in
$\mc{H}_i''$, which has coefficient zero in the bigon region that
does not contain the basepoint. Simply by adding the underlying
2-chains, these two domains $E^1$ and $E^2$ induce a domain $E$ in
$(\mc{H}_1*\mc{H}_2)''$ connecting $(0,x^1,y^1)$ to
$(0,x^2,y^2)$.  It follows from Lipshitz' Maslov index formula \cite{RL} that $\mu(E)=\mu(E^1)+\mu(E^2)$, and it is immediate that
$n_w(E)=n_w(E^1)+n_w(E^2)$. Consequently, the relative Maslov grading
is preserved. Therefore, in order to finish the proof, we only need to
calculate the absolute Maslov grading shift under the given isomorphism of relatively $\Z$-graded chain complexes. We will calculate this shift using the triangle maps associated to handleslides of the circles in $S^2$ over curves in the remainder of the diagram.

Towards this end, modify the Heegaard diagrams $\mc{H}''$ once more to
get the Heegaard diagrams $\mc{H}'''$ of the fourth row of
Figure~\ref{fig:modified-diagrams}. Namely, we slide the $\alpha$
circles inside $S^2$ off the attaching handles of $\Si^2$ (the ones
marked O in Figure~\ref{fig:modified-diagrams}) and we slide the
$\beta$ circles inside $S^2$ off the attaching handles of $\Si^1$ (the
ones marked X in Figure~\ref{fig:modified-diagrams}). There is an
obvious identification between $\alpha$ circles, $\beta$ circles, and
generators $(\vec{a},x^o)$ of ${\mc{H}''}$ and the corresponding
objects of ${\mc{H}'''}$; let $\bar\alpha$, $\bar\beta$, and $(\vec{a},\bar{x}^o)$ denote the corresponding objects in
${\mc{H}'''}$. Then each $\alpha$ circle only intersects the
corresponding $\bar\alpha$ circle, and does so at two points. So the
Heegaard diagram $(\Si,\al,\be,\bar\al)$ is of the type as described
in Section~\ref{sec:triangle-map}. The top generator $\theta$ has
coordinates just next the white dots of $\mc{H}''$ and $\mc{H}'''$,
and indeed, there is a small Maslov index zero triangular domain
connecting  $(\vec{0},x^o)$ and $(\vec{0},\bar{x}^o)$. Therefore,
the Maslov grading of $(\vec{0},x^o)$ is same as the Maslov grading of
$(\vec{0},\bar{x}^o)$. A similar argument, but with the roles of $\al$
and $\be$ reversed, proves that the Maslov grading is preserved under
the handleslides of the $\beta$-circles as well.

Now $\tCF_{\mc{H}'''}$ decomposes into $2^{n-1}$ direct summands, one
for each $\vec{a}\in\{0,1\}^{n-1}$. Moreover, the map
$(0,x^o)\mapsto(\vec{a},x^o)$ is an isomorphism between
$\tCF^W_{\mc{H}'''}[|\vec{a}|]$ and the summand corresponding to
$\vec{a}$.  Let $\mc{H}'''_d$ denote the Heegaard diagram destabilized
$n-1$ times, obtained from $\mc{H}'''$ by removing the $(n-1)$ $\alpha$
and $\beta$ circles that lie inside $S^2$, and the $(n-1)$ $w$-markings
enclosed by them. Then by Theorem~\ref{thm:main-invariance},
\begin{equation}\label{eq:3-vs-3d}
\tCF^W_{\mc{H}'''}[n-1]\cong\tCF_{\mc{H}'''_d}
\end{equation}
via the map $(\vec{0},x^o)\mapsto x^o$. 

Now the map $(x^1,x^2)\mapsto x^1\otimes x^2$ produces an
identification of chain groups following a similar but simpler
argument of Equation
\ref{eq:main-isomorphisms}:
\begin{equation}\label{eq:final-graded-isom}
  \tCF_{(\mc{H}_1*\mc{H}_2)'''_d}\cong\tCF_{(\mc{H}_1)'''_d}\otimes\tCF_{(\mc{H}_2)'''_d}.
\end{equation}
Moreover, this map is a relative Maslov grading preserving chain map, using a similar (but easier) localization principle argument to that above. However,
$(\mc{H}_1)'''_d$, $(\mc{H}_2)'''_d$, and $(\mc{H}_1*\mc{H}_2)'''_d$
are Heegaard diagrams for $S^3$ with $(l_1+\delta_1-n+1)$,
$(l_2+\delta_2-n+1)$, and $(l+\delta-n+1)$ basepoints, respectively;
therefore, by Theorem~\ref{thm:main-invariance} (and since
$(l_1+\delta_1-n)+(l_2+\delta_2-n)=(l+\delta-n)$), either side of the
equation has homology $\otimes^{l+\delta-n}(\F_2\oplus\F_2[-1])$, so the chain isomorphism
~\eqref{eq:final-graded-isom} preserves absolute Maslov grading as well.

Combining this with Equation~\eqref{eq:3-vs-3d} and the previous fact 
that corresponding generators in $\mc{H}''$ and $\mc{H}'''$ have equal
Maslov gradings, we conclude that the isomorphism from
Equation~\eqref{eq:main-isomorphisms} shifts gradings by $n-1$. Then
with the aid of Equations~\eqref{eq:trivial-isomorphism}
and~\eqref{eq:technical-isomorphism}, we arrive at the desired graded isomorphism
Equation~\eqref{eq:graded-isom}.
\end{proof}

We turn now to Theorem \ref{thm2}, which states that $\tautop$ of a Murasugi sum is maximal if and only if $\tautop$ of each summand is maximal.

\begin{proof}[Proof of Theorem \ref{thm2}]

    By Proposition~\ref{prop:tau-std-prop}, we may take mirrors and
  prove the following statement for $\taubot$: If $L$ is a Murasugi
  sum of links $L_1$ and $L_2$ along minimal index Seifert surfaces,
  then $\taubot(L_i)=-g(L_i)$ for all $i\in\{1,2\}$ if and only if
  $\taubot(L)=-g(L)$.

  We will continue from the previous proof, and re-use the same
  notation. Thanks to Corollary~\ref{cor:2-step-filtered} and
  Proposition~\ref{prop:tau-std-prop}, we only need to prove that, for
  both $i=1,2$, the inclusion
  \[
    \tSF_{\mc{H}_i,S_i}\hookrightarrow\tCF_{\mc{H}_i}
  \]
  induces a non-zero map on homology if and only if the inclusion
  \[
    \tSF_{\mc{H}_1*\mc{H}_2,S_1*S_2}\hookrightarrow\tCF_{\mc{H}_1*\mc{H}_2}
  \]
  induces a non-zero map on homology.

  Thanks to Equations~\eqref{eq:trivial-isomorphism}
  and~\eqref{eq:technical-isomorphism}, it is enough to prove the
  above for $(\mc{H}'',S')$, that is,
  \begin{equation}\label{eq:implication-to-show}
    \forall i\big(H_*(\tSF_{\mc{H}''_i,S'_i})\to H_*(\tCF_{\mc{H}''_i})\text{ is non-zero}\big)  \Leftrightarrow H_*(\tSF_{\mc{H}''_1*\mc{H}''_2,S'_1*S'_2})\to H_*(\tCF_{\mc{H}''_1*\mc{H}''_2})\text{ is non-zero}.
  \end{equation}

  Recall that we have nested subcomplexes
  \[
    \tSF_{\mc{H}'',S'}\hookrightarrow\tCF^W_{\mc{H}''}\hookrightarrow\tCF_{\mc{H}''}.
  \]
  We claim the map $\tCF^W_{\mc{H}''}\hookrightarrow\tCF_{\mc{H}''}$
  is injective on homology for $\mc{H}''=\mc{H}''_1$, $\mc{H}''_2$, or
  $\mc{H}''_1*\mc{H}''_2$. For each of the three diagrams, the
  homology of $\tCF^W_{\mc{H}''}$ is isomorphic, up to a grading
  shift, to the homology of the $(n-1)$-times destabilized diagrams
  $\mc{H}'''_d$ (recall Equation \eqref{eq:3-vs-3d}); let $\omega$ denote its rank. By counting the
  number of basepoints, in each of the three cases, the homology of
  $\tCF_{\mc{H}''}$ has rank $2^{n-1}\omega$. As before, the
  generators of $\tCF_{\mc{H}''}$ can be represented as pairs
  $(\vec{a},x^o)$ where $\vec{a}\in\{0,1\}^{n-1}$, and the
  differential is filtered with respect to $\vec{a}$. The associated
  graded complex of this filtration has $2^{n-1}$ summands, each
  isomorphic to $\tCF^W_{\mc{H}''}$. By a spectral sequence argument,
  the homology of the quotient complex
  $\tCF_{\mc{H}''}/\tCF^W_{\mc{H}''}$ has rank at most
  $(2^{n-1}-1)\omega$. Therefore, in the exact triangle
  \[
    \begin{tikzpicture}[xscale=3,yscale=0.8]
      \node (sub) at (0,0) {$H_*(\tCF^W_{\mc{H}''})$};
      \node (main) at (1,1) {$H_*(\tCF_{\mc{H}''})$};
      \node (quot) at (2,0) {$H_*(\tCF_{\mc{H}''}/\tCF^W_{\mc{H}''})$};
      \draw[->] (sub) -- (main);
      \draw[->] (main) -- (quot);
      \draw[->] (quot) -- (sub);
    \end{tikzpicture}
  \]
  the three terms have ranks $\omega$, $2^{n-1}\omega$, and at most
  $(2^{n-1}-1)\omega$, which implies that the map
  $H_*(\tCF^W_{\mc{H}''})\to H_*(\tCF_{\mc{H}''})$ is injective.
  
  Thanks to this, instead of Equation~\eqref{eq:implication-to-show}, it is enough to prove
  \[
    \forall i\big(H_*(\tSF_{\mc{H}''_i,S'_i})\to H_*(\tCF^W_{\mc{H}''_i})\text{ is non-zero}\big)  \Leftrightarrow H_*(\tSF_{\mc{H}''_1*\mc{H}''_2,S'_1*S'_2})\to H_*(\tCF^W_{\mc{H}''_1*\mc{H}''_2})\text{ is non-zero},
  \]
  which follows from Equation~\eqref{eq:main-isomorphisms}.
  \end{proof}
  
  \begin{rem} We remark that one direction of the theorem (maximality of $\tautop$ of a Murasugi sum implies maximality for its summands) could be deduced from the fact that maximality of $\tautop$ is preserved under taking subsurfaces of a minimal index Seifert surfaces. The latter fact is a consequence of a bound satisfied by $\tautop$ for cobordisms between links, analagous to \cite[Theorem 1]{HeddenRaoux}.\end{rem}

\bibliographystyle{myalpha}

\bibliography{mybibfile}

\end{document}